\documentclass[hidelinks,onefignum,onetabnum]{siamart220329}

\usepackage{enumerate}
\usepackage{braket,amsfonts,amsopn}
\usepackage{CJKutf8}
\usepackage{graphicx,epstopdf} 
\usepackage{yhmath}
\usepackage{mathrsfs}
\usepackage{mathtools}
\usepackage{float}
\usepackage{subfigure}
\usepackage{makecell}
\usepackage{multirow}

\usepackage{bbm}

\newcommand{\Star}{\mathrm{st}}
\newcommand{\VertSet}{\mathrm{Vert}}

\makeatletter
\newcommand{\fdsy@scale}{1.0}
\newcommand\fdsy@mweight@normal{Book}%
\newcommand\fdsy@mweight@small{Regular}%
\newcommand\fdsy@bweight@normal{Medium}%
\newcommand\fdsy@bweight@small{Bold}%
\DeclareFontFamily{U}{FdSymbolF}{}
\DeclareFontShape{U}{FdSymbolF}{m}{n}{
    <-7.1> s * [\fdsy@scale] FdSymbolF-\fdsy@mweight@small
    <7.1-> s * [\fdsy@scale] FdSymbolF-\fdsy@mweight@normal
}{}
\DeclareFontShape{U}{FdSymbolF}{b}{n}{
    <-7.1> s * [\fdsy@scale] FdSymbolF-\fdsy@bweight@small
    <7.1-> s * [\fdsy@scale] FdSymbolF-\fdsy@bweight@normal
}{}
\makeatother
\DeclareSymbolFont{delimiters}{U}{FdSymbolF}{m}{n}
\SetSymbolFont{delimiters}{bold}{U}{FdSymbolF}{b}{n}
\DeclareMathDelimiter{\lAngle}{\mathopen}{delimiters}{"92}{delimiters}{"92}
\DeclareMathDelimiter{\rAngle}{\mathclose}{delimiters}{"98}{delimiters}{"92}

\usepackage{lipsum}
\usepackage{amsmath}
\usepackage{amsfonts}
\usepackage{graphicx}
\usepackage{epstopdf}
\usepackage{enumerate}
\usepackage{algorithmic}
\usepackage{subfigure}
\usepackage{algorithm}
\usepackage{algorithmic} 
\usepackage{booktabs}
\usepackage{array}
\Crefname{ALC@unique}{Line}{Lines} 

\newcommand{\algorithmicprecondition}{ \textbf{Require:}}
\newcommand{\PreConditions}{\item[\algorithmicprecondition]}
\newcommand{\algorithmicpostcondition}{ \textbf{Side effect:}}
\newcommand{\PostConditions}{\item[\algorithmicpostcondition]}
\newcommand{\algorithmicinput}{ \textbf{Input:}}
\newcommand{\Input}{\item[\algorithmicinput]}
\ifpdf
  \DeclareGraphicsExtensions{.eps,.pdf,.png,.jpg}
\else
  \DeclareGraphicsExtensions{.eps}
\fi

\newcommand{\dif}{\mathrm{d}}
\newcommand{\Dim}{{ \textup{D}}}

\newcommand{\norm}[1]{\left\lVert#1\right\rVert}
\newsiamremark{remark}{Remark}
\newsiamremark{hypothesis}{Hypothesis}
\crefname{hypothesis}{Hypothesis}{Hypotheses}
\newsiamthm{claim}{Claim}
\newsiamthm{assumption}{Assumption}

\usepackage{amsopn}

\headers{Linear MARS method for 3D interface tracking}
{Yunhao Qiu and Qinghai Zhang}

\title{A linear MARS method for \\
  three-dimensional interface tracking\thanks{
    Qinghai Zhang is the corresponding author.
    \funding{This work was supported by
      the National Natural Science Foundation of China (\#12272346)
      and the Fundamental Research Funds
      for the Central Universities 226-2025-00254.}}}

\author{Yunhao Qiu\thanks{School of Mathematical Sciences,
Zhejiang University, Hangzhou, Zhejiang, 310058, China.
 (\email{qiuyunhao@zju.edu.cn}, \email{qinghai@zju.edu.cn}).}
\and Qinghai Zhang\footnotemark[1] \footnotemark[2]
\thanks{Institute of Fundamental and Transdisciplinary Research,
Zhejiang University, Hangzhou, Zhejiang, 310058,
China.}}

\ifpdf
\hypersetup{
  pdftitle={3D linear MARS method for explicit interface tracking of
  triangular meshes with consistent quality},
  pdfauthor={Yunhao Qiu and Qinghai Zhang}
}
\fi

\begin{document}

\maketitle

\begin{abstract}
  For explicit interface tracking in three dimensions,  
 we propose a linear MARS method that 
(a) represents the interface
 by a partially ordered set of glued surfaces
 and approximates each glued surface with a triangular mesh, 
(b) maintains an $(r,h,\theta)$-regularity on each triangular mesh 
 so that the distance between any pair of adjacent markers
 is within the range $[rh,h]$
 and no angle in any triangle is less than $\theta$, 
(c) applies to three-dimensional continua
 with arbitrarily complex topology and geometry,
(d) preserves topological structures and geometric features
 of moving phases under diffeomorphic and isometric flow maps, 
 and (e) achieves second-order and third-order accuracy
 in terms of the Lagrangian and Eulerian length scales, respectively.
Results of classic benchmark tests verify
 the effectiveness of the novel mesh adjustment algorithms
 in enforcing the $(r,h,\theta)$-regularity
 and demonstrate the high accuracy and efficiency 
 of the proposed linear MARS method.




\end{abstract}



\section{Introduction}
\label{sec:intro}
Numerical simulations of three-dimensional (3D) multiphase flows
are crucial
in both scientific research and engineering applications.
A main challenge in these simulations is
how to accurately track interfaces
that separate different fluid phases.
Various numerical methods have been developed for interface tracking (IT),
among which the most popular are volume-of-fluid (VOF) methods 
\cite{HiNi81, Be92, RiKo98, Be02, XiHo05, LiRu06,
  WaYa12, QiLu19, AsMe20, DeJo20, CoSp21},
level set methods
\cite{OsSe88, SuSm94, SuFa98, MoKe98, SuPu00, OsFe03, OlKr05, MiGi07,
  ThGi15, MiGu16, GiFe18},
and front tracking methods 
\cite{UnTr92, JuTr96, TrBu01, HuaSt08, JeRo15, ShLu18}.

In level set methods, 
the interface is \emph{implicitly} approximated as the zero isocontour
of a signed distance function.
In VOF methods,
the interface is not only \emph{implicitly} described by volume
fractions of the tracked phase inside fixed control volumes
but also \emph{explicitly} represented as a cellwise function.
Within each time step, a VOF method consists of two substeps:
in the first reconstruction substep the explicit representation of
material regions is determined \emph{solely} from volume fractions
while in the second advection substep the volume fractions are advanced
 from the explicit representation and the velocity field.
An extension of VOF is the moment-of-fluid (MOF) method \cite{AhSh09},
 which reconstructs the interface not only by volume fractions but
 also by centroids.
There also exist hybrid methods \cite{SuPu00, YaJa06}
 that aim to combine the advantages of level set and VOF methods.

Level set methods and VOF methods
 avoid topological and geometric problems
 by converting them into numerically
 solving partial differential equations
 such as the scalar conservation law 
 or the advection equation.  
This simplification accounts for much of their successes
 during the past decades. 
On the other hand,
 this simplification also leads to excessive numerical diffusion
 along the interface
 and consequently it is difficult for these methods
 to preserve
 geometric features of the interface
 even under isometric flow maps. 
For example,
 the Zalesak disk rotation \cite{zal79}
 is a popular benchmark test where the flow map induced from
 the purely rotational velocity field is an isometry; 
 but it is well known that level set methods and VOF methods
 fail to preserve the geometric features in that
 they round sharp corners of the Zalesak disk to circular arcs. 
Another related disadvantage of these methods is
 their lack of means to preserve topological invariants
 of moving phases under homeomorphic flow maps. 
As mentioned in \Cref{sec:objective} and shown in \Cref{fig:vortex}, 
 the flow map of an ordinary differential equation (ODE)
 $\frac{\dif \mathbf{x}}{\dif t}=\mathbf{u}(\mathbf{x},t)$
 that admits a unique solution is homeomorphic
 and thus preserves the number of connected components
 of a tracked phase. 
The flow map of the deformation test \cite{LeVe96}
 is such an example, cf. (\ref{eq:deformationVel}).
However, 
 level set methods and VOF methods tend to generate flotsams,
 thus increasing the number of connected components; 
 see \cite[Fig. 6 and Fig. 8]{BoSc22}
 for an illustration.


In contrast,
 front tracking methods are much better in preserving
 topological invariants and geometric features of the moving phases,
 thanks to the \emph{explicit} interface representation,
 which usually takes the form of a triangular mesh
 that consists of a set of markers,
 a set of edges describing the connectivity of the markers,
 and a set of triangles formed by adjacent edges.
Also, the explicit representation incurs little numerical diffusion
 and leads to more accurate IT than implicit representations
 \cite{GiFe18}.
 
Nonetheless,
 current front tracking methods have a number of serious limitations.

First,
 the interface markers may evolve
 to become unevenly distributed,
 leading to very long/short edges
 and small angles of triangles.
The long edges deteriorate the accuracy
 while the small angles 
 make the discrete system ill-conditioned.
To ensure an acceptable level of accuracy and stability, 
 certain regularity conditions,
 such as the $(r, h, \theta)$-regularity in \Cref{def:repinv}, 
 must be imposed on the evolving triangular mesh at each time step. 
%
Most existing 3D front tracking methods
 \cite{UnTr92, TrBu01, HuaSt08, JeRo15, ShLu18}
 primarily rely on elementary operations
 such as edge splitting, edge collapse, and edge flipping.
However, as discussed in \Cref{subsec:elementary}
 and shown in \cref{fig:swapfail} and
 \cref{fig:armadillocomparison}, 
 the improvement of the mesh quality
 offered by these elementary mesh adjustments 
 is rather limited 
 and cannot be satisfactory for all scenarios.

Second,
 front tracking methods lack rigorous foundations 
 such as a theory for modeling topological structures
 and geometric features of 3D continua. 
Can front tracking methods be applied to any 3D continua with arbitrarily
 complex topology and geometry?
What is the relationship between a 3D continuum
 and its ``front''? 
Will a triangular mesh \emph{suffice} as a data structure
 for representing any 3D continuum?
As far as we know,
 these fundamental questions have never been answered
 nor even mentioned in the literature of front tracking methods. 

Third,
 current front tracking algorithms are of limited accuracy,
 flexibility, and efficiency.
Due to its restriction that
 each Eulerian cell contains no more than
 one linear frontal segment \cite{UnTr92}, 
 the front tracking method is at best second-order accurate
 with no subgrid resolution. 
When two or more frontal segments are present in one cell,
 they must be merged or eliminated; 
 these surgical operations might be inconsistent
 with the bona fide physical processes 
 and may result in large errors.

Fourth, 
 it is sometimes important
 to efficiently obtain global topological invariants
 such as the number of connected components of a fluid phase
 and the number of holes inside the fluid; 
 for example, the number of bubbles per unit volume in a bubbly flow
 is a crucial characterization of the fluid properties.
However, it is highly difficult for front tracking methods,
 as well as level set methods and VOF methods, 
 to efficiently compute these global topological invariants. 

The above discussions motivate questions as follows.
\begin{enumerate}[({Q}-1)]
\item Based on the complete classification 
  of 3D continua in \mbox{\cite[Thm. 6.3]{ZhQi}}, 
  can we 
  uniquely and efficiently represent any 3D continuum
  with arbitrarily complex topology?
  In particular,
  can the global topological invariants of any 3D continuum
  be computed with optimal complexity, i.e., in O(1) time?
\item For a low-quality triangular mesh, 
  can we adjust the mesh efficiently and effectively
  so that it satisfies the $(r,h,\theta)$-regularity
  in \Cref{def:repinv}?
\item In light of the answers to (Q-1) and (Q-2), 
  can we develop an accurate and efficient MARS method
  for 3D explicit IT?
\end{enumerate}

In this paper, we give positive answers to all above questions.
Different from current methods that
 avoid topology and geometry in IT, 
 our primary principle is to \textit{tackle topological
 and geometric problems with tools in topology and geometry.}
Previously, we have designed the Yin space \cite{ZhLi20, ZhQi}
 as a mathematical model for two-dimensional (2D) and 3D continua,
 proposed the MARS
 (mapping and adjusting regular semianalytic sets) framework 
 for analyzing explicit IT methods \cite{ZhFo16},
 and developed fourth- and higher-order MARS methods
 for 2D two-phase IT \cite{Zh18}, 2D multiphase IT \cite{TaQi25}, 
 and 2D mean curvature flows \cite{HuLi25}.
We do not consider our MARS methods as a front tracking method,
 but rather a new family of methods in its own right.
In this work, we propose the first MARS method for 3D IT. 
 
In \Cref{sec:yinsets},
 we briefly review Yin sets as a mathematical model of 3D continua
 to lay the theoretical foundation of this work,
In particular,  
 we answer \mbox{(Q-1)}
 by \Cref{sec:sign-yin-space}.
In \Cref{sec:objective},
 we clearly state the IT problem to be solved.
In \Cref{sec:framework}, we review 
 the MARS framework to prepare for later sections.
In \Cref{subsec:3DlinearMARS}, 
 we answer (Q-3)
 by proposing a 3D linear MARS method, 
 of which the main novelty
 is a three-step sequence of mesh adjustment algorithms.

To resolve (Q-2), 
 we first adjust an ill-shaped triangular mesh
 by three elementary operations 
 discussed in \Cref{subsec:elementary}.
When they fail, we set up a spring system
 as the proxy of the triangular mesh
 and move the vertices
 to minimize the potential energy of the spring system
 so that edge lengths of the mesh
 become more evenly distributed;
 see \Cref{subsec:relocate}.
When this relocation algorithm fails again,
 our last resort is to replace a neighborhood of ill-shaped triangles
 with a newly generated local patch;  
 see more details in \Cref{subsec:regenerate}. 
This three-step sequence of mesh adjustment algorithms
 not only enforces
 the $(r_{\mathrm{tiny}},h_{L},\theta)$-regularity in IT,
 but also improves the smallest interior angle
 of a generic triangular mesh.
As such, it is the most important contribution of this work.

In \Cref{sec:tests},
 results of benchmark tests
 verify the effectiveness of the three-step sequence
 of mesh adjustment algorithms, 
 confirm the high accuracy and efficiency 
 of the proposed method, 
 and demonstrate its advantages over level set and VOF methods
 in preserving topological structures
 and geometric features of the moving phases.
In \Cref{sec:conclusions},
 we conclude this paper with future research prospects.



\section{Yin sets for modeling 3D continua}
\label{sec:yinsets}
In \Cref{sec:yin-sets}, we sum up
 relevant concepts and main conclusions in \cite{ZhQi}
 on the unique boundary representation of Yin sets. 
In \Cref{sec:triangular-mesh-as}, 
 the boundary of a Yin set is approximated 
 with a geometric 2-simplicial complex, 
 bridging the theory in \cite{ZhQi}
 and the algorithms in this work.
The significance of Yin sets for IT
 is summarized in \Cref{sec:sign-yin-space}.


\subsection{Uniquely representing Yin sets via glued surfaces}
\label{sec:yin-sets}

In a topological space ${\mathcal X}$,
 the \emph{complement} of a subset ${\mathcal P}\subseteq {\mathcal X}$,
 written ${\mathcal P}'$,
 is the set ${\mathcal X}\setminus {\mathcal P}$.
The \emph{closure} of a set ${\mathcal P}\subseteq{\mathcal X}$,
 written $\overline{\mathcal P}$,
 is the intersection of all closed 
 supersets of ${\mathcal P}$.
The \emph{interior} of ${\mathcal P}$, written ${\mathcal P}^{\circ}$,
 is the union of all open subsets of ${\mathcal P}$.
The \emph{exterior} of ${\mathcal P}$,
 written ${\mathcal P}^{\perp}:= {\mathcal P}^{\prime\circ}
 :=({\mathcal P}')^{\circ}$,
 is the interior of its complement.
A point $\mathbf{x}\in {\mathcal X}$ is
 a \emph{boundary point} of ${\mathcal P}$
 if $\mathbf{x}\not\in {\mathcal P}^{\circ}$
 and $\mathbf{x}\not\in {\mathcal P}^{\perp}$.
The \emph{boundary} of ${\mathcal P}$, written $\partial {\mathcal P}$,
 is the set of all boundary points of ${\mathcal P}$.
It can be shown that
 ${\mathcal P}^{\circ}={\mathcal P}\setminus \partial {\mathcal P}$
 and
 $\overline{\mathcal P}= {\mathcal P}\cup \partial {\mathcal P}$.

A \emph{regular open} set 
 is an open set ${\cal P}$ satisfying
 ${\mathcal P}=\overline{\mathcal P}^{\circ}$
 while a \emph{regular closed} set
 is a closed set ${\cal P}$
 satisfying ${\mathcal P}=\overline{\mathcal P^{\circ}}$.
Regular sets, open or closed,
 capture a key feature of continua
 that their regions are free of lower-dimensional elements
 such as isolated points and curves in $\mathbb{R}^2$
 and dangling faces in $\mathbb{R}^3$.
In this work, we employ regular open sets instead of regular closed sets
 because the latter does not have a \emph{unique} boundary
 representation
 \cite{ZhQi}
 while the former does, cf. \Cref{thm:bdryrep}. 
 
The intersection of two regular sets 
 might contain an infinite number of connected components
 \cite[eqn (3.1)]{ZhLi20},
 making it difficult to perform Boolean algorithms on regular sets 
 since no computer has an infinite amount of memory. 
This difficulty is resolved by requiring each regular set
 to be simultaneously a \emph{semianalytic} set,
 i.e., a set
  ${\cal P}\subseteq \mathbb{R}^{\Dim}$
 in the universe of 
 a finite Boolean algebra formed from the sets
 \mbox{${\cal X}_i
 =\left\{{\mathbf{x}}\in\mathbb{R}^{\Dim}:g_i({\mathbf{x}})\geq 0\right\}$}
 where each function $g_i:\mathbb{R}^{\Dim}\to\mathbb{R}$
 is analytic.
Intuitively,
 $\partial {\cal P}$ is piecewise smooth
 so that ${\cal P}$ can be described by a finite number of entities.

\begin{definition}[Yin space \cite{ZhFo16, ZhLi20}]
  \label{def:yinspace}
  A \emph{Yin set} $\mathcal{Y}\subset \mathbb{R}^{\Dim}$ is a regular
  open semianalytic set whose boundary is bounded.
  All Yin sets form the \emph{Yin space} $\mathbb{Y}$.
\end{definition}

Hereafter we focus 
 on the case of $\Dim=3$.
Instead of modeling the interface as a 2-manifold,
 we regard it as the boundary of a continuum represented by a Yin set.
Although this work does not address topological changes,
 the theory of Yin sets lays a solid foundation
 for future research on topological changes.

\begin{definition}
  \label{def:surface}
  A \emph{surface} is a piecewise-smooth orientable compact 2-manifold.
\end{definition}


The orientation of a surface ${\cal S}$ at $p\in {\cal S}$ is indicated
by a basis $(\mathbf{v}_1,\mathbf{v}_2)$ of the tangent space $T_p{\cal S}$,
we regard $\mathbf{v}_1$ and $\mathbf{v}_2$ as vectors in $\mathbb{R}^3$
by the embedding of $T_p{\cal S}$ in $T_p\mathbb{R}^3$
and define the \emph{induced normal vector} $\mathbf{n}$ at $p\in{\cal S}$
as the vector induced from the right-hand rule
to form a basis $(\mathbf{v}_1,\mathbf{v}_2, \mathbf{n})$ of $\mathbb{R}^3$. 
The surface ${\cal S}$ is \emph{positively oriented}
if $\mathbf{n}$ 
always points from its bounded complement 
to its unbounded complement;
otherwise it is \emph{negatively oriented}.

A \emph{one-dimensional (1D) CW complex}
 is the attaching space formed by gluing a number of disjoint simple curves
 to a finite number of points,
 with the two endpoints of each curve
 identified with one or two points.
In particular, a finite number of isolated points
 also form a 1D CW complex.

\begin{definition}
  \label{def:gluedSurface}
  A \emph{glued surface} is the quotient space
  of a surface such that 
  (i) the quotient map glues the surface
  along a 1D CW complex
  and (ii) the complement of the glued surface in $\mathbb{R}^3$
  has exactly two connected components,
  one bounded and the other unbounded.
\end{definition}

In light of \Cref{def:gluedSurface},
 the orientation of a glued surface
 is naturally given by that of the original surface.
 Then we have
 
\begin{definition}
\label{def:internalComplementOfGluedSurface}
  The \emph{internal complement of an oriented glued surface}
  $G\subset \mathbb{R}^3$,
 written $\mathrm{int}(G)$,
 is its bounded complement if $G$ is positively oriented;
 otherwise it is the unbounded complement. 
\end{definition}

Let $\mathcal{S}_1$ and $\mathcal{S}_2$ be
 two glued surfaces; 
 they are \emph{disjoint}
 if $\mathcal{S}_1\cap \mathcal{S}_2=\emptyset$. 
Otherwise a maximally connected component $\gamma$ of
 $\mathcal{S}_1\cap \mathcal{S}_2$ 
 is called an \emph{improper intersection}
 of $\mathcal{S}_1$ and $\mathcal{S}_2$
 if $(\mathcal{S}_2\cap \mathcal{N}(\gamma)) \backslash \gamma$
 is contained in a single component of
 $\mathcal{N}(\gamma) \backslash \mathcal{S}_{1}$
 in a sufficiently small open neighborhood
 $\mathcal{N}(\gamma)\subset \mathbb{R}^3$;
 otherwise $\gamma$ is a \emph{proper intersection}.
See \mbox{\cite[Fig. 3]{ZhQi}} for an illustration. 
Two glued surfaces are \emph{almost disjoint}
 if they have no proper intersections
 and their improper intersection
 is a 1D CW complex.

A glued surface $G_{k}$ is said
 to \emph{include} another glued surface $G_{\ell}$,
 written $G_{k} \ge G_{\ell}$ or $G_{\ell} \le G_k$,
 if the bounded complement of $G_{\ell}$
 is a subset of that of $G_k$.
If $G_k$ includes $G_{\ell}$ and $G_k \neq G_{\ell}$,
 we write $G_k > G_{\ell}$ or $G_{\ell} < G_k$.
In a partially ordered set (poset) $\mathcal{G}$ of glued surfaces
 with inclusion as the partial order,
 $G_k$ is said to \emph{cover} $G_{\ell}$ in $\mathcal{G}$,
 written $G_{k} \succ G_{\ell}$ or $G_{\ell}\prec G_{k}$, 
 if $G_{\ell}< G_{k}$ and no element $G\in {\mathcal G}$
 satisfies $G_{\ell}< G<G_{k}$.

\begin{theorem}[Unique boundary representation of connected 3D Yin sets \cite{ZhQi}]
  \label{thm:connectedYinsetrep}
  The boundary of any connected Yin set
  $\mathcal{Y} \neq \emptyset, \mathbb{R}^3$
  can be uniquely oriented and partitioned into a finite poset
  $\mathcal{G}_{\partial \mathcal{Y}}=\{G_j\subset \partial {\cal Y}\}$
  of pairwise almost disjoint oriented glued surfaces
  so that the following boundary representation of $\mathcal{Y}$
  is unique: 
  \begin{equation}
    \label{eq:boundaryRep3DYinSets}
    \mathcal{Y} = \mathop{\bigcap}\nolimits_{G_j \in\, \mathcal{G}_{\partial \mathcal{Y}}}
    \mathrm{int}(G_j),
  \end{equation}
  where $\mathrm{int}(G_j)$ is the internal complement
  of $G_j$ in \Cref{def:internalComplementOfGluedSurface}. 
Furthermore, the poset $\mathcal{G}_{\partial \mathcal{Y}}$
must be one of the following two types,
\begin{equation}
  \label{equ:decomptypes}
  \renewcommand{\arraystretch}{1.1}
  \left\{
    \begin{array}{ll}
      {\mathcal G}^-
      =\{G^-_1, G^-_2, \ldots, G^-_{n_-}\}
      & \text{where}\ n_-\ge 1,
      \\
      {\mathcal G}^+
      =\{G^+,G^-_1, G^-_2, \ldots, G^-_{n_-}\}
      & \text{where}\ n_-\ge 0,
    \end{array}
  \right.
\end{equation}
where all $G^-_j$'s are negatively oriented and 
mutually incomparable with respect to inclusion.
For ${\mathcal G}^+$,
 we have $G_j^- \prec G^+$, 
 i.e., $G^+$ covers $G^-_j$
 for each $j=1,2,\ldots,n_-$.
  \label{thm:bdryrep}
\end{theorem}

\Cref{thm:bdryrep} can be applied
 to each connected component of a Yin set ${\cal Y}$
 straightforwardly
 to obtain the unique representation of ${\cal Y}$;
 see \cite[Corollary 7.11]{ZhQi}.




\subsection{Approximating glued surfaces
  with triangular meshes}
\label{sec:triangular-mesh-as}

The points $\mathbf{x}_0, \mathbf{x}_1, \ldots, \mathbf{x}_m$
 in $\mathbb{R}^n$
 are \emph{affine independent} if
 $\{\mathbf{x}_1-\mathbf{x}_0, \mathbf{x}_2-\mathbf{x}_0, \ldots,
 \mathbf{x}_m-\mathbf{x}_0\}$
 is a linearly independent subset of $\mathbb{R}^n$.  
An $m$-\emph{simplex} $\mathbf{X}$
 or a \emph{simplex with dimension $m$}
 is the convex hull of $m+1$ affine independent
 points
 $\mathbf{x}_0, \mathbf{x}_1, \ldots, \mathbf{x}_m$, 
 written
 \mbox{$\mathbf{X} = [\mathbf{x}_0, \mathbf{x}_1, \ldots, \mathbf{x}_m]
 := \mathrm{conv} \{\mathbf{x}_0, \mathbf{x}_1, \ldots,
 \mathbf{x}_m\}$}
 and $\dim \mathbf{X}=m$; 
 then the \emph{vertex set of the $m$-simplex} $\mathbf{X}$
 is the affine independent set,
 written 
 $\mathrm{Vert}(\mathbf{X})
 = \{\mathbf{x}_0, \mathbf{x}_1, \ldots, \mathbf{x}_m\}$.
A \emph{face of a simplex} $\mathbf{X}$ is 
 a simplex $\mathbf{V}$ satisfying
 $\mathrm{Vert}(\mathbf{V})\subseteq \mathrm{Vert}(\mathbf{X})$.
 
A (geometric) \emph{simplicial complex} $K$
 is a finite set of simplexes in $\mathbb{R}^n$ such that
\begin{enumerate}[(a)]\itemsep0em
\item $\sigma\in K$ implies
  that all of its faces belong to $K$;
\item $\sigma,\tau\in K$ implies
  that $\sigma\cap \tau$ is either empty
  or a single common face of $\sigma$ and $\tau$.
\end{enumerate}

The \emph{dimension of a simplicial complex} $K$ is
 $\dim K := \sup_{\sigma\in K} \{\dim \sigma\}$. 
A \emph{simplicial $k$-complex} is a simplicial complex
 with $k=\dim K$.
The \emph{vertices}, \emph{edges}, and \emph{triangles}
 of a simplicial $2$-complex
 are its 0-simplices, 1-simplices, and 2-simplices, respectively.

The \emph{underlying space of a simplicial complex} $K$, 
 written $|K|$,
 is the union of all simplexes in $K$, 
 i.e., $|K| := \cup_{\sigma\in K} \sigma$. 
A topological space is \emph{triangularizable}
 or \emph{can be triangulated}
 if it is homeomorphic to $|K|$
 for some simplicial 2-complex $K$.
For example,
 it is well known that
 any compact orientable 2-manifold can be triangulated \cite{DoMo68}.
By \Cref{def:gluedSurface}, 
 a glued surface can also be triangulated. 
Then it follows from \Cref{thm:bdryrep} that 
 for any 3D Yin set ${\cal M}$
 there exists a simplicial 2-complex $K$
 such that $\partial{\cal M}$ is homeomorphic
 to $|K|$ and can be well approximated by $|K|$.

\begin{definition}
 \label{def:triangularMesh}   
 A \emph{triangular mesh} is a simplicial 2-complex $K$
 where each triangle is oriented so that,
 according to the right-hand rule, 
 the ordering of the three vertices 
 yields the outward normal direction of $|K|$ at this triangle.
\end{definition}

The (closed) \emph{star of a vertex $p$ in a triangular mesh $K$} 
 is a subset of $|K|$ defined as 
 $\overline{\Star}(p)
 := \bigcup_{\sigma\in K, p\in \VertSet(\sigma)} \sigma$
 and the \emph{link of $p$ in $K$}
 is the set containing every face $\tau\in K$
 such that $\tau$ and $p$ are disjoint
 and there exists a simplex in $K$
 having both $p$ and $\tau$ as its faces.
As the crucial difference between a surface and a glued surface, 
 the star of any vertex in the triangular mesh for the former
 must be homeomorphic to the unit disk
 while that for the latter needs not be so.

\subsection{Theoretical
   and algorithmic significance of Yin sets for IT}
\label{sec:sign-yin-space}


First,
\Cref{thm:bdryrep} furnishes
 a unique and efficient boundary representation of 3D continua
 with arbitrarily complex topology and geometry.
By the one-to-one correspondence between Yin sets 
 and posets of glued surfaces,
 the representation of any 3D continuum ${\cal Y}$ is reduced
 to those of the glued surfaces that constitute $\partial {\cal Y}$. 

Second,
 by \Cref{sec:triangular-mesh-as},
 there is no loss of generality
 in approximating a glued surface
 with a triangular mesh.
Hence in this work,
 any 3D Yin set ${\cal Y}$ is represented
 as a poset of triangular meshes,
 each of which approximates a glued surface in $\partial{\cal Y}$.
For each connected Yin set,
 the data structure for its representation
 corresponds to the two cases in (\ref{equ:decomptypes}). 
 
Lastly, it can be deduced from \Cref{thm:bdryrep}
 that the number of its connected components of a bounded Yin set
 equals the number of its positively oriented glued surfaces
 and that the number of holes inside a connected Yin set ${\cal Y}$, 
 i.e., Betti number $B_{2}$ of $\overline{{\cal Y}}$, 
 equals the number of negatively oriented glued surfaces in
 \cref{equ:decomptypes}. 
Furthermore,
 these characteristic numbers of global topology
 can be extracted in $O(1)$ time.




\section{The  IT problem}
\label{sec:objective}
Yin sets only model stationary continua.
To track a moving continuum,
 we follow the approach in \cite{ZhFo16}
 to formulate the IT problem as
 the determination of the image
 of the initial Yin set under the action of a flow map. 

\begin{definition}
  \label{def:it}
  The \emph{IT problem} is the determination of a Yin set
  $\mathcal{M}(T)\, \in \,\mathbb{Y}$ from its initial condition
  $\mathcal{M}(t_0)\,\in\, \mathbb{Y}$ under passive advection of
  a nonautonomous ordinary differential equation (ODE)
  $\frac{\dif \mathbf{x}}{\dif t} = \mathbf{u}(\mathbf{x}, t)$
  where $t$ is the time,
  $\mathbf{x}(t) \, \in \, \mathcal{M}(t)$
  the position of a passively advected Lagrangian particle in the Yin set,
  and $\mathbf{u}$ the velocity field that is continuous in time
  and Lipschitz continuous in space.
\end{definition}

The above ODE
 $\frac{\dif \mathbf{x}}{\dif t} = \mathbf{u}(\mathbf{x}, t)$
 admits a unique solution for any initial time and initial position.
This uniqueness furnishes a function
$\phi \, : \, \mathbb{R}^{\Dim} \times \mathbb{R}\times \mathbb{R}
\, \to\, \mathbb{R}^{\Dim}$,
\begin{equation}
  \label{equ:flowmap}
  \begin{array}{l}
  \phi_{t_0}^{\tau}(p(t_0))\, :=\, p(t_0+\tau) \,=\, p(t_0) +
  \int^{t_0+\tau}_{t_0} \mathbf{u}(p(t), t)\, dt,  \\
  \end{array}
 \end{equation}
 where the three variables $p(t_0)$, $t_0$, and $\tau$
 are the initial position, the initial time, and the time increment,
 respectively.
If both the initial time and the time increment are fixed, 
 the flow map $\phi$ is a diffeomorphism.



\section{The MARS framework}
\label{sec:framework}

On top of its topological equivalence classes and Boolean algebra \cite{ZhLi20}, 
 the Yin space is further augmented to a metric space
 by the function
 $d \,: \, \mathbb{Y} \times \mathbb{Y} \rightarrow [0, +\infty)$
 given by 
\begin{equation}
  \label{equ:metric}
  \forall \mathcal{P}, \mathcal{Q} \in \mathbb{Y},\quad
  d(\mathcal{P}, \, \mathcal{Q})\, := \, \Vert \mathcal{P} \oplus
  \mathcal{Q} \Vert, 
\end{equation}
where the \emph{volume of a Yin set} $\mathcal{Y}$ is defined as 
 $\Vert \mathcal{Y} \Vert := \left|\int_{\mathcal{Y}}d
   \mathbf{x}\right|$,
 the \emph{regularized union} ``$\cup^{\perp\perp}$''
 as ${\cal P}\cup^{\perp\perp} {\cal Q}:=
 \left(({\cal P}\cup{\cal Q})^{\perp}\right)^{\perp}$, 
 and the \emph{regularized symmetric difference} ``$\oplus$'' as
 \begin{equation}
   \label{eq:symmetricDiff}
   \mathcal{P} \oplus \mathcal{Q} := \left( \mathcal{P}\,\backslash\,
     \mathcal{Q} \right) \,
   \mathop{\cup}\nolimits^{\perp\perp} \,
   \left( \mathcal{Q} \,\backslash \,\mathcal{P} \right).
 \end{equation}
Note that $d$ is indeed a metric on $\mathbb{Y}$
because, for any $\mathcal{P}, \mathcal{Q},{\cal R} \in \mathbb{Y}$, 
${\cal P}\ne{\cal Q}$ implies $d({\cal P},{\cal Q})>0$,
$d({\cal P}, {\cal P})=0$, 
$d({\cal P}, {\cal Q}) = d({\cal Q}, {\cal P})$,
and
$d({\cal P}, {\cal Q}) \le d({\cal P}, {\cal R})
+ d({\cal R}, {\cal Q})$.

Then it is natural to define
 the total error of an IT method as 
 \begin{equation}
   \label{equ:totalerror}
   E_{\mathrm{IT}}(t_n) := d\left( \mathcal{M}(t_n), \mathcal{M}^n\right)
   = \Vert \mathcal{M}(t_n) \oplus \mathcal{M}^n \Vert
\end{equation}
where $\mathcal{M}^n$ and $\mathcal{M}(t_n)$
 are the computational and exact results of a moving Yin set
 at time $t_n$, respectively. 
A generic method for the IT problem in \cref{def:it}
is 

\begin{definition}
  \label{def:MARS} 
  (\cite{ZhFo16})
  A \emph{MARS} method is a generic 
  $\mathrm{IT}$ method $\mathfrak{L}_{\mathrm{MARS}} \,:\,
  \mathbb{Y}\rightarrow \mathbb{Y}$ that consists of three
  unitary operations at each time step
  \begin{equation}
    \label{equ:MARS}
    \mathcal{M}^{n+1}=\mathfrak{L}_{\mathrm{MARS}}^n \mathcal{M}^n\,:=\,
    \left(\chi_{n+1} \circ \varphi_{t_n}^k \circ \psi_n\right) \mathcal{M}^n,
  \end{equation}
  where $\mathcal{M}^n\in \mathbb{Y}$ is the computational result
  that approximates the exact solution
  $\mathcal{M}(t_n)\in \mathbb{Y}$
  of the $\mathrm{IT}$ problem in \cref{def:it},
  $\varphi_{t_n}^k : \mathbb{Y} \rightarrow \mathbb{Y}$
  a fully discrete flow map 
  that approximates the exact flow map $\phi_{t_n}^k$ 
  both in time and in space,
  $\psi_n : \mathbb{Y}\rightarrow \mathbb{Y}$
  an augmentation operation at $t_n$
  to prepare $\mathcal{M}^n$ for $\varphi_{t_n}^k$, and
  $\chi_{n+1} : \mathbb{Y} \rightarrow \mathbb{Y}$
  an adjustment operation after $\varphi_{t_n}^k$.
\end{definition}

It is shown in \cite[Sec. 3.2]{HuLi25} that
 the IT error of a MARS method is bounded
 by the sum of four individual errors
 that correspond to the approximation
 of the initial Yin set $\mathcal{M}(t_0)$
 and the operations in \cref{equ:MARS}.

Although all operations in \cref{equ:MARS} have the same signature 
 $\mathbb{Y} \rightarrow \mathbb{Y}$, 
 \cref{thm:bdryrep} suggests that
 each operation be considered as acting
 upon the glued surfaces of $\partial {\cal M}$.
For a MARS method to be practical,
 each glued surface is approximated
 with a homeomorphic triangular mesh.
We use $\psi_n$ and $\chi_{n+1}$ 
 to add and remove markers on the interface,
 respectively. 
Besides increasing and decreasing the number of markers,
 $\psi_n$ and $\chi_{n+1}$
 may also change the structure of the triangular mesh.
In contrast,
 $\varphi_{t_n}^k$ always preserves the number of markers
 and the structure of the triangular mesh. 

A 2D Yin set is represented 
 by sequences of markers on its boundary Jordan curves
 and a curve fitting scheme that maps a marker sequence
 to a continuous curve. 
To achieve fourth- or higher-order accuracy, 
 numerical stability 
 of solving the linear system for coefficients of piecewise polynomials
 is ensured by 
 the \emph{$(r,h)$-regularity} \mbox{\cite[Def. 4.12]{TaQi25}}
 of each marker sequence, 
 i.e., the distance between each pair of adjacent markers
 being within the range $[rh,h]$ where $h>0$ and $r\in(0,1]$. 
Note that
 the meaning of $(r,h)$-regularity differs 
 for periodic and not-a-knot splines; 
 see \cite[Def. 4.14 and Def. 4.17]{TaQi25}.
In both cases,
 the $(r,h)$-regularity 
 is guaranteed by
 the multiphase cubic MARS method proposed in \cite{TaQi25}
 and yields a proof of the fourth-order convergence 
 of the 2D MARS method \cite[Sec. 5]{TaQi25}. 
A straightforward generalization of the $(r,h)$-regularity from 2D to 3D
 gives 

\begin{definition}[$(r, h, \theta)$-regularity]
  \label{def:repinv}
  For $r\in(0,1]$, $h>0$, and $\theta>0$, 
  a triangular mesh is said to be \emph{$(r,h,\theta)$-regular}
  if each interior angle of every triangle in the mesh is no less than $\theta$
  and the length of any edge is
  within the range $\left[rh, h\right]$.
\end{definition} 

When we approximate a 3D Yin set with a set of triangular meshes, 
 there is no need to worry about
 ill conditioning of this boundary representation
 because no linear systems are solved.
Consequently, 
 for the sake of IT,
 only the $(r, h)$-regularity is needed
 to guarantee the second-order convergence of a MARS method.
 
However,
 IT is seldom a stand-alone problem
 and the $(r, h, \theta)$-regularity might be desirable
 in complex scenarios such as 
 fourth- and higher-order boundary representations of 3D Yin sets
 and the coupling of IT with a main flow solver
 to solve moving boundary problems. 
Indeed, the main focus of this work
 is the maintenance of the $(r, h, \theta)$-regularity
 for triangular meshes
 that represent the evolving Yin set.

\section{Algorithms}
\label{sec:algorithm}
A 3D linear MARS method is proposed
 in \Cref{subsec:3DlinearMARS},
 where the $(r_{\mathrm{tiny}}, h_{L}, \theta)$-regularity
 of the evolving triangular mesh
 is maintained at each time step 
 by a cascade of local mesh adjustments.  
The first, and also the simplest,  
 consists of three classical techniques
 reviewed in \Cref{subsec:elementary}. 
When these classical techniques fail,
 we improve the mesh regularity
 by moving vertices
 to minimize the potential energy of a spring system, 
 as detailed in \Cref{subsec:relocate}.
The last, and also the most general, 
 is to regenerate a local mesh
 from randomly scattered vertices
 in \Cref{subsec:regenerate}. 
The three families of mesh adjustments
 complement each other 
 in terms of efficiency, complexity, and generality.

\subsection{The 3D linear MARS method}
\label{subsec:3DlinearMARS}

By \Cref{thm:bdryrep}, 
 any 3D Yin set can be represented by 
 a poset of pairwise almost disjoint glued surfaces, 
 each of which
 can be approximated arbitrarily well by a triangular mesh, 
 cf. \Cref{def:triangularMesh}.
The data structure for this boundary representation
 is designed as a poset of triangular meshes,
 each of which is stored
 as a 3-tuple $\mathcal{T}=(V,E,T)$ 
 where $V=\{p_i\}$ is the vertex set  
 containing position vectors
 of interface markers in $\mathbb{R}^3$,
 $E=\{\overline{p_ip_j}\}$ the edge set
 consisting of unordered pairs of vertex indices in $V$,
 and $T=\{\triangle_{p_ip_jp_k}\}$ the triangle set 
 consisting of ordered triples of vertex indices in $V$.



\begin{definition}
  \label{def:3DlinearMARS}
  Let $\varphi$ be the discrete flow map that approximates the exact flow
  map $\phi$ of an IT problem.
  Let $h_L$, $r_{\mathrm{tiny}}$ and $\theta$ be
  the parameters in \cref{def:repinv}.
  In each time step $\left[ t_n, t_n+k \right]$,
  the \emph{3D linear MARS method}
  takes as its input the discrete representation
  of $\partial\mathcal{M}^n=\{{\cal T}_j^n\}$, and
  advances it to that of $\partial\mathcal{M}^{n+1}=\{{\cal T}_j^{n+1}\}$:
  \begin{enumerate}[({M3D}-1)]
  \item Initialize $\partial\mathcal{M}^{n+1}$
    by tracing forward in time vertices of $\partial\mathcal{M}^{n}$,
    i.e.,
    \begin{equation}
      \label{eq:flowMapOnMn}
      \forall {\cal T}_j^n\in\partial\mathcal{M}^{n},\quad 
      {\cal T}_j^{n+1} = \left(\varphi_{t_n}^k(V({\cal T}_j^n)),
      E({\cal T}_j^n), T({\cal T}_j^n)\right).
    \end{equation}
  \item If $\partial \mathcal{M}^{n+1}$
    has any edge 
    whose length is greater than $h_L$,
    \begin{enumerate}
    \item for each edge $\overline{p_ip_j}$ with $|p_i-p_j|>h_L$, 
      locate $\overleftarrow{p_i},\overleftarrow{p_j}$
      in $\partial \mathcal{M}^{n}$ as preimages of $p_i,p_j$
      and divide $\overline{\overleftarrow{p_i}\overleftarrow{p_j}}$
      into equidistant subedges by inserting new markers
      as new vertices of $\partial \mathcal{M}^{n}$,
    \item add new edges between vertex preimages
      and the newly added markers 
      so that $\partial \mathcal{M}^{n}$ remains a simplicial 2-complex,
    \item update $\partial \mathcal{M}^{n+1}$ by
      (\ref{eq:flowMapOnMn}).
    \end{enumerate}
    Repeat the above steps until no edge length
    in $\partial \mathcal{M}^{n+1}$ is greater than $h_L$.      
  \item For each edge in $\partial \mathcal{M}^{n+1}$ that is shorter than
    $r_{\mathrm{tiny}}h_L$, collapse the edge into one of its endpoints.
    Repeat this step until no edge is shorter than
    $r_{\mathrm{tiny}}h_{L}$.
  \item For each triangle in $\partial \mathcal{M}^{n+1}$
    that violates the $\theta$-regularity,
    \begin{enumerate}
    \item find a neighboring triangle
      where the common edge is adjacent to the smallest interior angle,
      combine the two triangles to form a quadrilateral,
      and check whether flipping its diagonal increases
      the smallest interior angle; 
      if so, flip it, 
    \item if (a) fails to achieve the $\theta$-regularity,
      apply \cref{alg:relocating}, 
    \item if (b) fails to achieve the $\theta$-regularity,
      apply  \cref{alg:regenerating}.
    \end{enumerate}
  \end{enumerate} 
\end{definition}

The above 3D linear MARS method is indeed a MARS method
 in the sense of \cref{def:MARS}: 
 substeps (M3D-2a,b) constitute
 the augmentation operation $\psi_n$
 that adds new markers to $\partial \mathcal{M}^n$
 to maintain the upper bound $h_L$ of edge lengths,
 substeps (M3D-1, 2c) constitute the discrete flow map
 that advects the interface markers,
 and substeps (M3D-3, 4) constitute the adjustment operation $\chi_{n+1}$
 that 
 enforces the lower bound $\theta$ of interior angles
 as well as the lower bound $r_{\mathrm{tiny}}h_L$
 of edge lengths.

Although it suffices to only employ the Lagrangian grid
 of a moving triangular mesh to evolve the interface, 
 an Eulerian grid is often needed
 when coupling an IT method with a main flow solver.
One distinguishing feature of the MARS method
 is the user-specified relation $h_L = O\left(h^{\alpha}\right)$
 that connects the length scale $h_L$ of the interface
 to the length scale $h$ of the bulk flow
 such as the size of the Eulerian grid.
Since the interface is a set of codimension one
 and the CFL condition requires $O\left(\frac{1}{h}\right)$ time steps
 for a march in time, 
 the complexity of a 3D MARS method 
 is $O\left(\frac{1}{h^{1+2\alpha}}\right)$
 for the choice of $h_L = O\left(h^{\alpha}\right)$.
In contrast,
 the optimal complexity of a 3D main flow solver 
 is $O\left(\frac{1}{h^4}\right)$.
Therefore, 
 a 3D MARS method with $\alpha\le \frac{3}{2}$
 does not increase the complexity of the entire solver;
 see \cite[Sec. 5.2.4]{Zh18} for more discussions.
On the other hand,
 the proposed linear 3D MARS method
 is third-order accurate for 
 the choices of a third-order time integrator
 and the relation $h_L = O\left(h^{\frac{3}{2}}\right)$.

\subsection{Elementary mesh adjustments (EMA)}
\label{subsec:elementary}

If the $(r_{\mathrm{tiny}}, h_L, \theta)$-regularity is violated, 
 we first try some simple and efficient mesh adjustments,
 which are 
\begin{itemize}
\item edge splitting:
  any edge longer than $h_L$ is split at its midpoint; 
\item edge collapse:
  any edge shorter than $r_{\mathrm{tiny}}h_{L}$
  is collapsed into an endpoint; 
\item edge flipping:
  see (M3D-4a) in \Cref{def:3DlinearMARS}. 
\end{itemize}

\begin{figure}
\centering  
\subfigure[an edge splitting]{ 
\label{fig:basicops:sub1}
\includegraphics[width=0.53\textwidth]{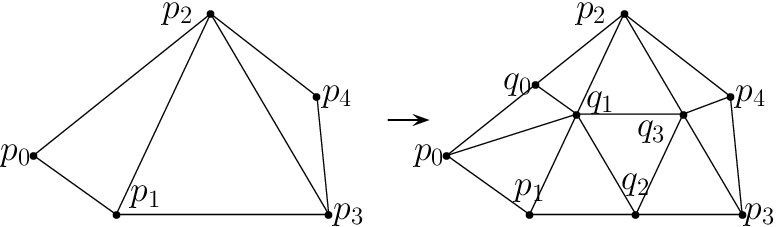}}
\hfill
\subfigure[an edge collapse]{
\label{fig:basicops:sub2}
\includegraphics[width=0.35\textwidth]{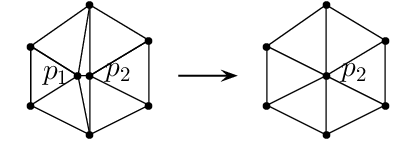}}

\subfigure[a successful edge flipping]{
\label{fig:basicops:sub3}
\includegraphics[width=0.46\textwidth]{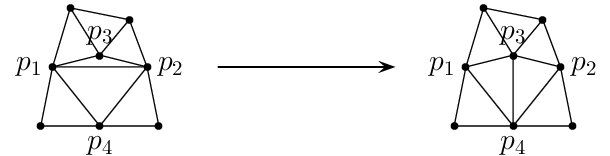}}
\hfill
\subfigure[an unsuccessful edge flipping]{
  \label{fig:swapfail}
\includegraphics[width=0.46\textwidth]{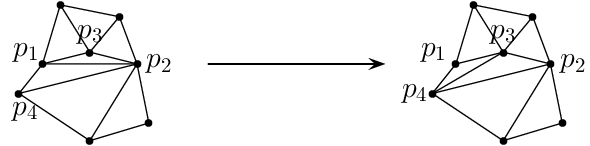}}
\caption{Three elementary mesh adjustments.
  In subplot (a),
  edges $p_{0}p_{2}$, $p_1p_2$, $p_2p_3$ and $p_1p_3$
  are longer than $h_L$ and are split at their midpoints; 
  new edges are then added locally to obtain a new triangular mesh.
  In subplot (b), the edge $p_1p_{2}$ shorter than $r_{\mathrm{tiny}}h_{L}$
  is collapsed into its endpoint $p_2$.
  In subplot (c),
  $\angle p_{3}p_1p_{2}$ and $\angle p_3p_2p_{1}$
  are smaller than $\theta$
  and the diagonal $p_1p_2$ is flipped to $p_3p_{4}$
  to remove the two small angles.
  In subplot (d), the edge flipping
  fails to fulfill the $\theta$-regularity.
 }
\label{fig:basicops}
\end{figure}  

The above three adjustments are schematically
 depicted in \cref{fig:basicops} (a,b,c).
Edges that are too long or too short
 are effectively handled by
 edge splitting or edge collapse, respectively. 
Edge flipping, however, may fail in certain cases;
 see \cref{fig:basicops} (d) for an example.
In Subsections \ref{subsec:elementary} and \ref{subsec:relocate},
 we deal with triangles that violate the $\theta$-regularity
 and to which the above edge flipping does not apply.
Neither increasing the maximum edge length
 nor decreasing the minimum edge length,
 the algorithms
 in Subsections \ref{subsec:elementary} and \ref{subsec:relocate}
 preserve the $(r_{\mathrm{tiny}},h_{L})$-regularity
 achieved by (M3D-2,3). 

\subsection{Vertex relocation
  via energy minimization (VREM)}
\label{subsec:relocate}

In a triangular mesh ${\cal T}=(V,E,T)$,
 very small angles tend to generate largely varying edge lengths.
This motivates us to move the vertices 
 so that edge lengths in ${\cal T}$ are as uniform as possible.
In \Cref{sec:minim-potent-energy},
 we associate a spring system with ${\cal T}$ 
 so that its total potential energy 
 measures the variance of the edge lengths in ${\cal T}$.
In \Cref{sec:projection-operator}, we define a projection operator
 to facilitate the movement of vertices of ${\cal T}$
 on $|{\cal T}|$,
 the underlying space of ${\cal T}$
 as defined in \Cref{sec:triangular-mesh-as}.
In \Cref{sec:vrem-algorithm},
 we propose a VREM algorithm 
 to minimize the potential energy 
 and prove its convergence. 

\subsubsection{Minimizing the total potential energy
  of a spring system}
\label{sec:minim-potent-energy}
 
An edge in $E$ is
 a \emph{boundary edge} if it is adjacent to
 (i.e., a face of) only one triangle in $T$;
 otherwise it is an \emph{interior edge}.
A vertex is a \emph{boundary vertex}
 if it is incident to any boundary edge;
 otherwise it is an \emph{interior vertex}.
The sets of boundary edges and interior edges
 are denoted by $E_B$ and $E_I$,
 respectively;
 similarly,
 those of boundary vertices and interior vertices
 are denoted by $V_B$ and $V_I$, 
 respectively.

Consider a triangular mesh $\mathcal{T}$
 whose boundary vertices are \emph{fixed} and
 whose interior vertices are \emph{free to move}. 
By definition a boundary edge
 only consists of boundary vertices
 and can thus be thought of as a metal bar.
In contrast, each interior vertex is viewed as a particle
 and each interior edge $\overline{p_ip_j}$ a spring.
Indeed, for $|p_i-p_j|$ to vary, 
 $p_i$ or $p_j$ must be an interior vertex.
 
By Hooke's law,
 the \emph{force} exerted by the edge $\overline{p_ip_j}$
 on the $i$th vertex is 
\begin{equation}
  \label{equ:force-i}
  \begin{array}{c}
    \forall \overline{p_ip_j}\in E_{I}({\cal T}),\quad 
    \mathbf{f}_{j}(p_i) =
    \frac{|p_{j} - {p}_{i}|-R_L}{|{p}_{j}-{p}_{i}|}
    \ ({p}_{j}-{p}_{i}), 
  \end{array}
\end{equation}
where $p_i=(x_i,y_i,z_i)^{\top}$ is the position vector
 of the $i$th vertex,
 $R_L$ the resting length of the spring $\overline{p_ip_j}$, 
 and $|p_j - p_i|=\sqrt{(x_i-x_j)^2+(y_i-y_j)^2+(z_i-z_j)^2}$
 the actual length of $\overline{p_i p_j}$. 
Although $R_L$ can be made a function of the interface curvature at $p_i$,
 in this work we simply set $R_L$
 to be the average length of all interior edges in ${\cal T}$.
According to (\ref{equ:force-i}),  
 the force between $p_i$ and $p_j$ is attractive if $|p_j-p_i|>R_L$
 and is repulsive if $|p_j-p_i|<R_L$.
As shown in \cref{fig:spring}, 
 the net force on $p_i$ tends to move it 
 so as to reduce the length variance of all edges at $p_i$.

\begin{figure}
\centering
\label{fig:spring}
\includegraphics[width=0.60\textwidth]{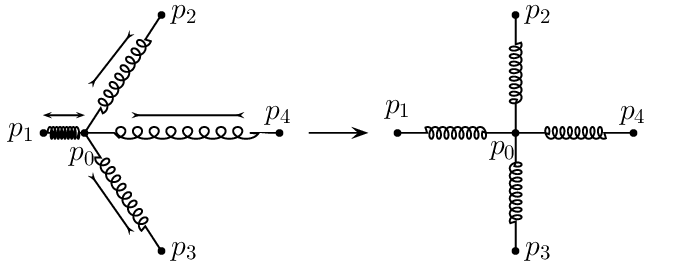}
\caption{The net force of a vertex $p_0$
  moves $p_0$ so that
  the lengths of edges at $p_0$ become closer to the resting length. 
  On the left, only $p_0p_1$ exerts a repulsive force on $p_0$
  while all other edges exert attractive forces on $p_0$. 
}
\end{figure}

By the work-energy theorem \cite[Sec. 6.2]{YoFr15}
 and the force $\mathbf{f}_j$ in (\ref{equ:force-i}) being conservative, 
 the \emph{potential energy of an interior edge} $\overline{p_ip_j}$
 can be defined as 
 \begin{equation}
   \label{equ:Eij}
   \begin{array}{c}
     U_{ij} = \int_{ {p}_{i}}^{{q}_{i}}
     \mathbf{f}_{j}(q) \cdot\dif q
     =-\int_{|{p}_{j} - {p}_{i}|}^{R_L}
     (r - R_L) \mathrm{d}r 
     =\frac{1}{2} (|{p}_{j} - {p}_{i}| - R_L)^{2}, 
   \end{array}
 \end{equation}
 where the point $q_i$ satisfies $|q_i-p_j|=R_L$, 
 the three points $q_i,p_i,p_j$ are collinear,
 and $r$ is the length of $\overline{q p_j}$
 as $q$ moves from $p_i$ to $q_i$.
The second equality of (\ref{equ:Eij})
 follows from the variable substitution
 $\dif r = \frac{q-p_j}{|q-p_j|}\dif q$.
Clearly, the potential energy $U_{ij}$
 is always non-negative
 and is zero only when the edge length equals the resting length.
It is then natural to define
 the \emph{total potential energy of a triangular mesh} ${\cal T}$ as 
\begin{equation}
  \label{eq:potentialEnergyTri}
  U({\cal T}) := 
  \sum\nolimits_{\overline{p_i p_j}\in E_I} U_{ij}.
\end{equation}

By (\ref{equ:Eij}),
 $U({\cal T})$ is non-negative and is zero only if
 all interior edges in ${\cal T}$ have the same length $R_L$. %
In terms of its interior vertices,
$U({\cal T})$ can also be expressed as
\begin{equation}
  \label{eq:potentialEnergyVert}
  U({\cal T}) = U(\mathbf{p})= \frac{1}{2}  
  \sum\nolimits_{\overline{p_i p_j}\in E_I}
  (|p_j - p_i| - R_L)^{2},
\end{equation}
where $\mathbf{p} = (p_1^{\top}, p_2^{\top}, \ldots, p_N^{\top})^{\top}$ and
$N$ is the number of interior vertices. 
By (\ref{equ:force-i}), (\ref{equ:Eij}),
(\ref{eq:potentialEnergyTri}), and (\ref{eq:potentialEnergyVert}), 
we have 
 \begin{equation}
   \label{eq:energyGradient}
   \frac{\partial U}{\partial p_i} :=
   \left(\frac{\partial U}{\partial x_i},
     \frac{\partial U}{\partial y_i},
     \frac{\partial U}{\partial z_i}\right)^{\top} =
   -\sum\nolimits_{\overline{p_i p_j}\in E_I}
   \mathbf{f}_j(p_i) =: -F_i,
 \end{equation}
 where the last equality in \cref{eq:energyGradient}
 defines $F_i$ as \emph{the net force acting on the $i$th vertex}. 
 Then the energy gradient $\nabla U$ is defined by
 \begin{equation}
  \label{eq:nablaU}
  \nabla U^{\top} := \left(\frac{\partial U}{\partial p_1}^{\top},
    \frac{\partial U}{\partial p_2}^{\top}, \cdots,
  \frac{\partial U}{\partial p_N}^{\top}\right). 
\end{equation}

To sum up,
 the variance of edge lengths in a triangular mesh ${\cal T}$
 is decreased by reducing $U({\cal T})$. 
Therefore, we aim to solve the minimization problem
\begin{equation}
  \label{eq:minimizationPE}
  \min_{V_I} U({\cal T})
  \quad \text{subject to }
  \left\{
  \begin{array}{l}
    \text{the topology of } {\cal T} \text{ remaining unchanged};
    \\
    \text{positions of the vertices staying on } |{\cal T}|, 
  \end{array}
  \right.
\end{equation}
where the first condition precludes
mesh entanglements such as coinciding vertices
and crossing edges
while the second condition helps to accurately approximate
 $|{\cal T}|$ 
 with the underlying space of the adjusted triangular mesh.

We would like to convert (\ref{eq:minimizationPE}) 
 to an unconstrained optimization problem. 
This is done by first ignoring the second condition
 in (\ref{eq:minimizationPE})
 to allow free 
 movements of interior vertices
 and then projecting their new positions
 back to $|{\cal T}|$.
 
\subsubsection{A local projection 
  to $|{\cal T}|$ at $p$}
\label{sec:projection-operator}

The free movement of an interior vertex $p$
 is made sufficiently small
 so that the image of its new position
 under the local projection at $p$ 
 is within $\overline{\Star}(p)$,
 the star of $p$ in $|{\cal T}|$; see \Cref{sec:triangular-mesh-as}. 

\begin{definition}
  \label{def:localProj}
  The \emph{local projection} $\mathfrak{P}^{\mathcal{T}}_{p}$ 
  of a triangular mesh ${\cal T}$ at an interior vertex $p$ 
  maps points in a neighborhood $\mathcal{N}(p)\subset\mathbb{R}^3$
  to $\overline{\Star}(p)\subset |{\cal T}|$, 
  by steps as follows. 
  \begin{enumerate}[({LPT}-1)]
  \item Denote by $I_{\text{adj}}(p):=
    \{i: p_i=(x_i,y_i,z_i)^{\top} \text{ is adjacent to } p \}$
    the index set of vertices adjacent to $p$ 
    and suppose 
    $\max_{i,j\in I_{\text adj}(p)}|z_i - z_j|$, 
    the coordinate span in $z$,
    is no greater than those in $x$ and in $y$.
    Fit a plane $\Gamma$ in the form of $z=Ax+By+C$ to the vertices
    $(p_j)_{j\in I_{\text adj}(p)}$ 
    by solving 
    \begin{equation}
      \label{eq:leastsquare}
      \begin{array}{l}
        \min\limits_{A, B, C}\,\sum\limits_{j\in I_{\text adj}(p)}
        (Ax_j + By_j+C-z_j)^{2}.
      \end{array}
    \end{equation}
    In similar cases of the coordinate span in $x$ or in $y$
    being the smallest, the form of $\Gamma$ 
    is $x=Ay+Bz+C$ or $y=Ax+Bz+C$, respectively.
  \item Define a projector ${\cal P}_{\Gamma}:
    {\cal N}(p)\rightarrow \Gamma$
    by setting ${\cal P}_{\Gamma}(q)$
    to be the point that is closest to $q$ in $\Gamma$.
    Then apply ${\cal P}_{\Gamma}$
    to all triangles in $\overline{\Star}(p)$
    to obtain a planar triangular mesh.
  \item For a point $q\in\mathcal{N}(p)$,
    determine which planar triangle contains ${\cal P}_{\Gamma}(q)$,
    and compute its barycentric coordinates with respect to that triangle.
  \item Use the barycentric coordinates to compute the corresponding
    point on the original triangle, which is returned
    as $\mathfrak{P}^{\mathcal{T}}_{p}(q)$. 
  \end{enumerate}
  When there is no danger of ambiguity,
  we simply write $\mathfrak{P}_{p}$ instead of $\mathfrak{P}^{\mathcal{T}}_{p}$. 
\end{definition}

It is well known that the least-squares problem (\ref{eq:leastsquare})
 is solved by applying the QR factorization
 to the linear system $M\mathbf{c}=\mathbf{z}$
 where the $i$th row of $M$ is $(x_i,y_i,1)$,
 $\mathbf{c}=(A,B,C)^{\top}$,
 and the $i$th component of $\mathbf{z}$ is $z_i$. 
As $h_L\rightarrow 0$,
 the underlying space $|{\cal T}|$,
 the fitted plane $\Gamma$, 
 and the local projection $\mathfrak{P}_{p}$ converge to
 the tracked 2-manifold, 
 the tangent plane of the 2-manifold at $p$, 
 and the identity operator, respectively.

The neighborhood ${\cal N}(p)$ 
 is implicitly defined by (LPT-2,3):
 ${\cal N}(p)$ consists of each point $q$
 whose image under ${\cal P}_{\Gamma}$
 falls within some planar triangle in \mbox{(LPT-2)}.
%
As an affine transformation,
 the orthogonal projector ${\cal P}_{\Gamma}$
 preserves barycentric coordinates.
Consequently, 
 the local projection
 $\mathfrak{P}_p: {\cal N}(p) \to \overline{\Star}(p)$ satisfies
 $\mathfrak{P}_p(q)=q$ for any $q\in \overline{\Star}(p)$; 
 then $\overline{\Star}(p)\subset {\cal N}(p)$ implies
 that $\mathfrak{P}_p$ is indeed a projection,
 i.e., $\mathfrak{P}_p^2=\mathfrak{P}_p$. 

\subsubsection{The VREM algorithm and its convergence}
\label{sec:vrem-algorithm}

\begin{figure}
\centering
\label{fig:projection}
\subfigure[the net force of $p$ moves $p$ to $p'$]
{\includegraphics[width=0.40\textwidth]{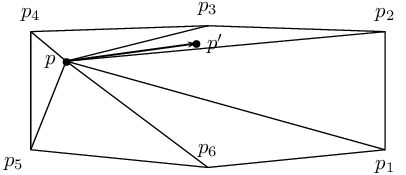}}
\hfill
\subfigure[project $p'$ onto $\overline{\Star}(p)$ to obtain $p''$]
{\includegraphics[width=0.40\textwidth]{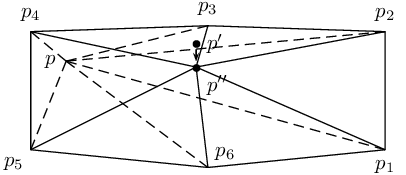}}
\caption{The VREM algorithm updates
  positions of the interior vertices of a triangular mesh
  by first moving each $p_i\in V_I$
  along the direction of its net force to $p'$
  and then projecting $p'$ onto the star of $p$
  to obtain $p''$ as the updated position of $p$.
  As defined in \Cref{sec:triangular-mesh-as}
  and shown in subplot (a), 
  the star of $p$ is the union of the six triangles
  with $p$ as their common vertex 
  while the link of $p$ is 
  the simple closed curve formed by the line segments
  `` $p_{1}\rightarrow p_{2} \rightarrow p_{3}
  \rightarrow p_{4} \rightarrow p_{5}
  \rightarrow p_{6} \rightarrow p_{1}$.''
}
\end{figure} 

As shown in \Cref{fig:projection},
the core of our solution to (\ref{eq:minimizationPE}) is

\begin{definition}
  \label{def:VREM}
  Given a triangular mesh ${\cal T}$
  with its interior vertex set $V^{(s)}_I$ at the $s$th iteration, 
  the \emph{VREM algorithm} 
  computes an updated vertex set $V^{(s+1)}_I$ by 
  \begin{equation}
    \label{eq:VREM}
    \forall p_i^{(s)}\in V_I^{(s)},\quad
    {p}^{(s+1)}_i= \mathfrak{P}_{p_i^{(0)}}  \left(p^{(s)}_i
      + \alpha^{(s)} F_i^{(s)}\right), 
  \end{equation}
  where $F_i^{(s)}$ is given
  by (\ref{eq:energyGradient}) and (\ref{equ:force-i}),
  the local projection $\mathfrak{P}_{p_i^{(0)}}$
  is described in Definition \ref{def:localProj}, 
  and the step length $\alpha^{(s)}$ is determined by
  \begin{enumerate}[(SPL-1)]
  \item initializing $\alpha^{(s)}$
    with $\alpha^{(s)}_0 := \min\limits_{i}
    \frac{2\ell_{i}^{(s)}}{5\|F_i^{(s)}\|_2}$
    where $\ell_{i}^{(s)}$ is the minimum distance
    from $p_i^{(s)}$ to any point in its link, 
    cf. the last paragraph in \Cref{sec:triangular-mesh-as}.
  \item iteratively reducing $\alpha^{(s)}$ by
    $\alpha^{(s)}_{k+1} = \rho \alpha^{(s)}_k$ until
    \begin{equation}
      \label{eq:backtrackingCond}
      \begin{array}{l}
        U(\mathbf{p}^{(s+1)})\le U(\mathbf{p}^{(s)})
        + c \alpha^{(s)} \nabla U(\mathbf{p}^{(s)})^{\top} \mathbf{d}^{(s)},
      \end{array}
    \end{equation}
    where $\rho, c\in (0,1)$
    are user-defined parameters
    and the vector $\mathbf{d}^{(s)}$ is 
    \begin{equation}
      \label{eq:ds}
      \begin{array}{rl}
        (\mathbf{d}^{(s)})^{\top} &:=
        \left( d_{1}^{(s)\top},  d_{2}^{(s)\top},\cdots ,
        d_{N}^{(s)\top}\right),
        \\
        d_{i}^{(s)} &:= \frac{1}{\alpha_{0}^{(s)}}
                      \left(\mathfrak{P}_{p_i^{(0)}}(p_i^{(s)}+\alpha_{0}^{(s)} F_i^{(s)})
                      - p_i^{(s)}\right). 
      \end{array}
    \end{equation}
  \end{enumerate}
\end{definition}

The line search (\ref{eq:VREM}) reduces to the steepest descent
 in the asymptotic range of $h_L\rightarrow 0$, 
 because $\mathfrak{P}_{p_i^{(0)}}$ converges to the identity
 and the descent direction $F_i=-\frac{\partial U}{\partial p_i}$
 in (\ref{eq:energyGradient}) implies that
 $\mathbf{d}^{(s)}$ converges to $-\nabla U(\mathbf{p}^{(s)})$
 in (\ref{eq:backtrackingCond}). 
Then we have
\begin{displaymath}
  U\left(\mathbf{p}^{(s)} - \alpha^{(s)} \nabla U^{(s)}\right) =
  U(\mathbf{p}^{(s)}) - \alpha^{(s)} \|\nabla U^{(s)}\|^2
  + O((\alpha^{(s)})^2).
\end{displaymath}

On the one hand,
 $\alpha^{(s)}$ should be sufficiently small so that
 the total potential energy $U$ is indeed reduced, i.e., 
 $U(\mathbf{p}^{(s)} - \alpha^{(s)} \nabla U^{(s)})< U(\mathbf{p}^{(s)})$.
On the other hand,
 $\alpha^{(s)}$ should be large enough so that
 the reduction of $U$ is substantial in each VREM iteration.
In (SPL-1), 
 the maximum offset of any $p^{(s)}_i$
 is limited to $\frac{2}{5}\ell^{(s)}_i$
 to prevent potential entanglement of the triangular mesh,
 fulfilling the first condition in (\ref{eq:minimizationPE}). 
Then the initial value of $\alpha^{(s)}$ is repeatedly reduced
 until (\ref{eq:backtrackingCond}) holds.
This simple approach yields
 a suitable value of $\alpha^{(s)}$ that substantially reduces $U$.
In this work, we use $(c,\rho)=(10^{-4},0.8)$.

Given a triangle $\triangle_{p_i p_j p_k}$ that violates
 the $(r_{\mathrm{tiny}}, h_{L}, \theta)$-regularity,
 we employ a breadth-first search (BFS)
 to identify a triangular mesh $\mathcal{T}$
 that covers $\triangle_{p_i p_j p_k}$, 
 then we apply to ${\cal T}$ the VREM algorithm in \Cref{def:VREM}.
We formalize the algorithmic steps of VREM
 in \Cref{alg:relocating}
 and prove its convergence in \Cref{lem:VREM}. 

\begin{algorithm}
\caption{\texttt{VREM}}
\label{alg:relocating} 
\begin{algorithmic}[1]
  \Input{${\cal T}_{\partial {\cal M}}$: 
    a triangular mesh that represents the boundary $\partial \mathcal{M}$
    of a Yin set ${\cal M}$;
    \\ \hspace{5mm}
    $(r_{\mathrm{tiny}}, h_{L}, \theta)$:
    regularity parameters;
    \\ \hspace{5mm}
    $\triangle_{p_i p_j p_k}$: a triangle  
    in ${\cal T}_{\partial \mathcal{M}}$ that
    violates the $(r_{\mathrm{tiny}}, h_{L}, \theta)$-regularity;
    \\ \hspace{5mm}
    $\mu$: the maximum number of BFSs;
    \\ \hspace{5mm}
    $\nu$: the maximum number of iterations of VREM.
  }
  \PreConditions $h_L$ is sufficiently small.
  \PostConditions {${\cal T}_{\partial {\cal M}}$
    is updated with a more uniform distribution of edge lengths.}
  \STATE{$\mathcal{T} \leftarrow$ an empty triangular mesh;}
  \STATE{$V \leftarrow$ a vertex set $\{p_{i}\}$,
    where the interior angle at $p_i$
    is the largest in $\triangle_{p_i p_j p_k}$;}
  \STATE{isRegular $\leftarrow$ False, iter $\leftarrow$ 1;} 
  \WHILE{isRegular == False
    \text{and} iter $<$= $\mu$}
  \STATE{Add to $\mathcal{T}$
    all triangles in ${\cal T}_{\partial {\cal M}}$
    that are adjacent to at least one vertex in $V$;}
  \STATE{$V \leftarrow$ the vertex set of $\mathcal{T}$;}
  \STATE{Set the resting length as the average of all edges in ${\cal T}$;}
  \FOR{$\ell=1$ to $\nu$}   
  \STATE{Update the interior vertices of ${\cal T}$ by (\ref{eq:VREM});}
\IF{$\mathcal{T}$ satisfies
  the $(r_{\mathrm{tiny}},h_{L},\theta)$-regularity}
\STATE{Update the corresponding subcomplex
  in ${\cal T}_{\partial \mathcal{M}}$
  with the new $\mathcal{T}$;}
  \STATE{isRegular $\leftarrow$ True;}
  \STATE{Exit the loop;}
  \ENDIF
  \ENDFOR
  \STATE{iter $\leftarrow$ iter+1;}
  \ENDWHILE
\end{algorithmic}
\end{algorithm} 

\begin{lemma}
  \label{lem:VREM}
  The VREM algorithm converges on
  any triangular mesh ${\cal T}$ of which
  the maximum edge length is sufficiently small,
  \begin{equation}
    \label{eq:VREMconverges}
    \forall p_i\in V_I({\cal T}),\quad
    \lim_{s\rightarrow +\infty}\left(p_i^{(s+1)}-p_i^{(s)}\right)
    =\mathbf{0}.
  \end{equation}
\end{lemma}
\begin{proof}
  As $h_L\to 0$, $\mathfrak{P}_{p_i^{(0)}}$ and $d_i^{(s)}$ in (\ref{eq:ds})
  converge to the identity and $F_i^{(s)}$, respectively; 
  then in \cref{eq:backtrackingCond}
  the angle $\beta^{(s)}$ between $\mathbf{d}^{(s)}$
  and $-\nabla U(\mathbf{p}^{(s)})$ converges to 0.
  Hence there exists $h_L^*>0$ such that for each $h_L<h_L^*$
   we have $\cos \beta^{(s)} > \delta > 0$.
  
  (SPL-1) precludes the coincidence
   of any two relocated vertices, 
   hence $U$ is $C^{2}$
   and $\nabla U$ is Lipschitz continuous
   with its Lipschitz constant $L>0$.
  Then the vector version of the fundamental theorem of calculus yields
  \begin{displaymath}
    \begin{array}{rl}
    & U(\mathbf{p}^{(s)}+\alpha \mathbf{d}^{(s)}) - U(\mathbf{p}^{(s)})
    =  \int_0^\alpha  \nabla U(\mathbf{p}^{(s)}
      + t \mathbf{d}^{(s)}) ^\top \mathbf{d}^{(s)} \, \dif t
      \\
    =& \alpha \nabla U(\mathbf{p}^{(s)})^\top \mathbf{d}^{(s)} 
       + \int_0^\alpha \left[ \nabla U(\mathbf{p}^{(s)}
       + t \mathbf{d}^{(s)})
       - \nabla U(\mathbf{p}^{(s)}) \right]^\top \mathbf{d}^{(s)}
       \, \dif t
      \\
    \le& \alpha \nabla U(\mathbf{p}^{(s)})^\top \mathbf{d}^{(s)} +
         \int_0^\alpha L t \|\mathbf{d}^{(s)}\|^2 \, \dif t. 
    \end{array}
  \end{displaymath}

  Consequently, we have 
  \begin{equation}
    \label{eq:quad}
    \forall \alpha \in \left[0, \alpha_0^{(s)}\right],\ 
    U(\mathbf{p}^{(s)}+\alpha \mathbf{d}^{(s)})\le U(\mathbf{p}^{(s)}) +
    \alpha\nabla U(\mathbf{p}^{(s)})^{\top}\mathbf{d}^{(s)}
    + \frac{1}{2}L\alpha^2\|\mathbf{d}^{(s)}\|^2. 
  \end{equation}

By \cref{def:VREM},
 the step length $\alpha^{(s)}$ is
 either the initial step length $\alpha_0^{(s)}$
 or obtained by reducing a previous step length
 that violates \cref{eq:backtrackingCond}.
We only prove the latter case
since the arguments for the former case are similar.

If $\alpha^{(s)} < \alpha_0^{(s)}$,
then the previous step must have
violated \cref{eq:backtrackingCond}:
\begin{equation}
\label{eq:failed_armijo}
    \begin{array}{l}
U\left(\mathbf{p}^{(s)}
  + \frac{1}{\rho}\alpha^{(s)}\mathbf{d}^{(s)}\right)
> U(\mathbf{p}^{(s)}) +
\frac{c}{\rho} \alpha^{(s)}\nabla U(\mathbf{p}^{(s)})^{\top} \mathbf{d}^{(s)}.
    \end{array}
\end{equation}

Set $\alpha = \frac{1}{\rho}\alpha^{(s)}$ in \cref{eq:quad},
 and we have, from \cref{eq:failed_armijo}, 
$$
\begin{array}{c}
\frac{c}{\rho} \alpha^{(s)}\nabla U(\mathbf{p}^{(s)})^{\top}
  \mathbf{d}^{(s)}
  <
\frac{1}{\rho} \alpha^{(s)}\nabla U(\mathbf{p}^{(s)})^{\top}\mathbf{d}^{(s)}
  +
  \frac{L}{2\rho^{2}}\alpha^{(s)2}\|\mathbf{d}^{(s)}\|^2, 
  \end{array}
  $$
  which yields
$\alpha^{(s)} > 2\rho \frac{c-1}{L}
\frac{\nabla U(\mathbf{p}^{(s)})^{\top} \mathbf{d}^{(s)}}{\norm{\mathbf{d}^{(s)}}^2}$.
%
This lower bound of $\alpha^{(s)}$
and \cref{eq:backtrackingCond} give
\begin{displaymath}
  U(\mathbf{p}^{(s)}) - U(\mathbf{p}^{(s+1)}) \ge
  c \left( 2\rho \frac{1-c}{L}
    \frac{\nabla U(\mathbf{p}^{(s)})^{\top} \mathbf{d}^{(s)}}
    {\norm{\mathbf{d}^{(s)}}^2} \right)
  \nabla U(\mathbf{p}^{(s)})^{\top} \mathbf{d}^{(s)}.
\end{displaymath}
Replace $\nabla U(\mathbf{p}^{(s)})^{\top} \mathbf{d}^{(s)}$ with
$-\cos\beta^{(s)} \norm{\nabla U(\mathbf{p}^{(s)})}
\norm{\mathbf{d}^{(s)}}$ 
and we have
\begin{equation}
\label{eq:final_inequality}
U(\mathbf{p}^{(s)}) - U(\mathbf{p}^{(s+1)}) \ge
\frac{2c \rho (1-c)}{L} \cos^2\beta^{(s)} \norm{\nabla U(\mathbf{p}^{(s)})}^2.
\end{equation}

Write $C_{\rho} := \frac{2c \rho (1-c)}{L}>0$.
Summing \cref{eq:final_inequality} over iterations from $s=0$ to $S$
 yields
 \begin{displaymath}
U(\mathbf{p}^{(0)}) - U(\mathbf{p}^{(S+1)}) =
\sum_{s=0}^S (U(\mathbf{p}^{(s)}) - U(\mathbf{p}^{(s+1)})) \ge
  C_{\rho} \sum_{s=0}^S \cos^2\beta^{(s)} \norm{\nabla U(\mathbf{p}^{(s)})}^2.
 \end{displaymath}
Since $U$ is bounded below by 0, we have $U(\mathbf{p}^{(S+1)}) \ge 0$.
Thus, for any $S$:
$$
U(\mathbf{p}^{(0)})  \ge C_{\rho} \sum_{s=0}^S \cos^2\beta^{(s)}
\norm{\nabla U(\mathbf{p}^{(s)})}^2.
$$

As $S \to \infty$, $U(\mathbf{p}^{(0)})$ remains constant.
Hence the monotonically increasing series on the right-hand side
 must converge, implying 
$\sum_{s=0}^{\infty} \cos^2\beta^{(s)}
\norm{\nabla U(\mathbf{p}^{(s)})}^2 < \infty$.
Since $\cos\beta^{(s)} > \delta$, we have
$
\sum_{s=0}^{\infty} \norm{\nabla U(\mathbf{p}^{(s)})}^2 < \infty,
$
which yields
$
\norm{\nabla U(\mathbf{p}^{(s)})} \rightarrow 0.
$
The proof is then completed by \cref{eq:VREM}.
\end{proof}

\subsection{Local triangulation regeneration (LTR)}
\label{subsec:regenerate}

Although VREM converges,
 the first condition in (\ref{eq:minimizationPE})
 implies that the adjusted mesh
 be homeomorphic to the input mesh $\mathcal{T}$. 
As shown in \cref{fig:regenerate}(a),
 if $\mathcal{T}$ has a pathological arrangement, 
 this homeomorphism may lead to a failure of VREM
 in fulfilling the $\theta$-regularity. 
This limitation of VREM is resolved
 by replacing ${\cal T}$
 with a regenerated triangular mesh $\mathcal{T}'$
 that has a simpler topology as well as a suitable number of vertices.

The algorithmic steps of LTR are formalized in \cref{alg:regenerating}. 
In lines 1--6,
 the input triangle violating
 the $(r_{\mathrm{tiny}}, h_{L}, \theta)$-regularity
 is combined with some neighboring triangles 
 to form a triangular mesh $\mathcal{T}$.
In lines 7--8, we fit a plane $\Gamma$ to all vertices of $\mathcal{T}$
 and project the boundary edges of ${\cal T}$ onto $\Gamma$
 to obtain a planar polygon $\Lambda$.
In lines 12--13, we scatter $m^{*}$ points randomly inside $\Lambda$
 and obtain $\mathcal{T}''$ as
 the Delaunay triangulation of the vertices of $\Lambda$
 and the scattered points.
This step comes from the well-known fact
 that Delaunay triangulation maximizes
 the smallest angle in the triangular mesh
 \cite[Thm. 9.9]{BeCh10}. 
In lines 14--17, 
 vertices of $\mathcal{T}''$ are projected onto
 $|\mathcal{T}|$ in the same manner as that in (LPT-3,4); 
 this updates the vertex positions of ${\cal T}''$
 but preserves its topology. 
In line 18, we employ VREM in \cref{alg:relocating}
 to further improve the quality of the regenerated
 triangulation $\mathcal{T}'$.
To alleviate the random effects in point scattering,
 we repeat the above procedure multiple times
 and retain the regenerated triangular mesh
 with the largest minimum angle; 
 see lines 19--21
 and the two nested ``for'' loops starting at lines 10 and 11.
If the mesh remains unsatisfactory 
 after finishing the outermost iteration in lines 5--27,
 we identify a larger triangulation by the BFS
 and repeat the LTR steps
 until the maximum number $\mu$ of BFS
 is exceeded or
 the $(r_{\mathrm{tiny}}, h_{L}, \theta)$-regularity is satisfied. 
 
The positive integer $m_{\mathrm{est}}^*$ in line 9
 is determined as follows.
Denote by $(x_i, y_i)_{i=1}^{m}$
  vertices of the planar polygon $\Lambda$.
In the ideal case of
 $\Lambda$ being triangulated into equilateral triangles, 
 each interior vertex is surrounded by six equilateral triangles, 
 occupying a third of the area of each triangle.
Hence the area occupied by each interior vertex is
 $6\times \frac{1}{3}\times \frac{\sqrt{3}}{4}h_{\triangle}^2
 =\frac{\sqrt{3}}{2}h_{\triangle}^2$
 where $h_{\triangle}$ is the edge length; 
 see \cref{fig:regenerate}(c).
In contrast, 
 the area occupied by a boundary vertex 
 depends on the interior angle $\gamma_i$ of the polygon
 at that vertex, and is calculated by 
 $\frac{1}{3} \times \frac{\gamma_i}{\frac{\pi}{3}}  
 \times\frac{\sqrt{3}}{4}h_{\triangle}^2$.
 Since the interior angles sum up to $\sum_{i=1}^m \gamma_i=(m-2)\pi$,
 the total area occupied by the $m^*$ scattered vertices is
 \begin{displaymath}
   \begin{array}{l}
     S_c := \sum\limits_{i=1}^m \frac{1}{3}
     \times\frac{\gamma_i}{\frac{\pi}{3}}
     \times \frac{\sqrt{3}}{4} h_{\triangle}^2
     + m^* \times \frac{\sqrt{3}}{2} h_{\triangle}^2
     =  \frac{\sqrt{3}}{4}(m-2) h_{\triangle}^2
     +  \frac{\sqrt{3}}{2}m^*  h_{\triangle}^2.     
   \end{array}
 \end{displaymath}
On the other hand,
 the actual area $S_{\Lambda}$ of the polygon $\Lambda$ is given by Green's formula
 \begin{displaymath}
 \begin{array}{c}
 S_{\Lambda} := \frac{1}{2} \left\vert \sum\limits_{i=1}^m
    (x_i + x_{i+1})(y_{i+1} -  y_i )\right\vert,
   \end{array}
 \end{displaymath}
where $(x_{m+1},y_{m+1}) = (x_1, y_1)$.
Then $S_c = S_{\Lambda}$ yields an accurate estimate of $m^{*}$ as 
 \begin{equation}
   \label{eq:m-estimated}
   \begin{array}{l}
     m_{\mathrm{est}}^* := \max\left\{0, 
     1+\text{round}\left(\frac{2}{ \sqrt{3}h_{\triangle}^2}S_{\Lambda}
     - \frac{1}{2}m\right)\right\}, 
   \end{array}
 \end{equation}
 where $\text{round}(x)$ returns the integer closest to $x$
 and, in practice, $h_{\triangle}$ is set to
 the average edge length of $\Lambda$.

Consider the example in \cref{fig:regenerate}(a). 
For $\mathrm{iter}=1$,
 ${\cal T}$ at line 5 is
 the submesh inside the inner rectangle, 
 for which no regenerated mesh is satisfactory. 
For $\mathrm{iter}=2$, ${\cal T}$ is expanded 
 to include all triangles in \cref{fig:regenerate}(a)
 and $m^{*}=2$ leads to 
 the regenerated mesh in \cref{fig:regenerate}(b)
 with a simpler topology and fewer vertices.
In all our numerical tests,
 \cref{alg:regenerating} always succeeds 
 for $\theta=\frac{\pi}{10}$ with $(\mu, \nu, \eta)= (4, 10, 3)$. 
 

\begin{figure}
\centering
\label{fig:regenerate}
\subfigure[a triangular mesh]
{\includegraphics[width=0.32\textwidth]{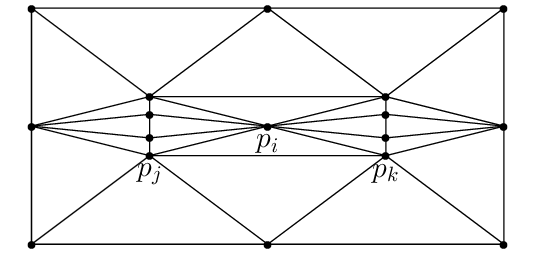}}
\subfigure[the regenerated mesh]
{\includegraphics[width=0.32\textwidth]{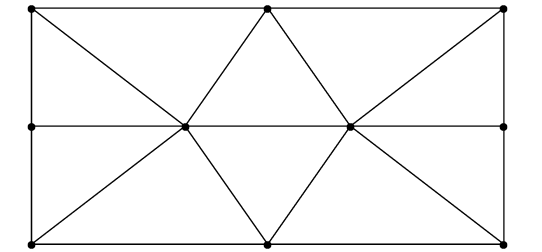}}
\subfigure[occupied areas of vertices]{
\includegraphics[width=0.30\textwidth]{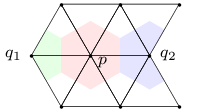}}
\caption{Illustrations of LTR.
  In (a),
  the pathological arrangement
  leads to an unsatisfactory result of VREM with
  $\mu=4$ and $\nu=10$.
  In (b),
  LTR with $\mu=4$, $\nu=10$, $\eta = 3$, and $\theta=\frac{\pi}{10}$
  generates a new triangulation with a simpler topology
  and fewer vertices.
  In (c), the green, red, blue
   regions represent the areas occupied by vertices
   $q_1$, $p$ and $q_2$, respectively.
}
\end{figure}

\begin{algorithm}
\caption{\texttt{LTR}}
\label{alg:regenerating} 
\begin{algorithmic}[1] 
  \Input{
    ${\cal T}_{\partial {\cal M}}$: 
    a triangular mesh that represents the boundary $\partial \mathcal{M}$
    of a Yin set ${\cal M}$;
    \\ \hspace{5mm}
    $(r_{\mathrm{tiny}}, h_{L}, \theta)$:
    regularity parameters;
    \\ \hspace{5mm}
    $\triangle_{p_i p_j p_k}$: a triangle  
    in ${\cal T}_{\partial \mathcal{M}}$ that
    violates the $(r_{\mathrm{tiny}}, h_{L}, \theta)$-regularity;
    \\ \hspace{5mm}
    $\mu$: the maximum number of BFSs;
    \\ \hspace{5mm}
    $\nu$: the maximum number of iterations of VREM; 
    \\ \hspace{5mm}
    $\eta$: the maximum number of randomly scattering points.}
  \PreConditions{ $h_L$ is sufficiently small.}
  \PostConditions{${\cal T}_{\partial {\cal M}}$
    is updated with a more uniform distribution of edge lengths.
    }
  \STATE{$\mathcal{T} \leftarrow$ an empty triangular mesh;}
  \STATE{$V \leftarrow$ a vertex set $\{p_{i}\}$,
    where the interior angle at $p_i$
    is the largest in $\triangle_{p_i p_j p_k}$;}
  \STATE{isRegular $\leftarrow$ False, iter $\leftarrow$ 1;} 
  \WHILE{isRegular == False
      \text{and} iter $<$= $\mu$}
  \STATE{Add to $\mathcal{T}$
    all triangles in ${\cal T}_{\partial {\cal M}}$
    that are adjacent to at least one vertex in $V$;}
  \STATE{$V \leftarrow$ the vertex set of $\mathcal{T}$;}
  \STATE{Fit a plane $\Gamma$ to the vertices of $\mathcal{T}$
   in the same manner as that in (LPT-1);}
 \STATE{$\Lambda \leftarrow {\cal P}_{\Gamma}(\partial|{\cal T}|)$
   where 
   ${\cal P}_{\Gamma}$ is the projector defined in (LPT-2);}
  \STATE{Initialize another triangular mesh
    $\mathcal{T}' \leftarrow \mathcal{T}$;}
  \FOR{$m^{*}=\max\{m_{\mathrm{est}}^*-2,0\}$ to
  $m_{\mathrm{est}}^*+2$
  with $m_{\mathrm{est}}^*$ given in (\ref{eq:m-estimated})}
\FOR{$j = 1$ to $\eta$}
\STATE{Randomly scatter $m^{*}$ points
  inside the planar polygon $\Lambda$;}
\STATE{$\mathcal{T}'' \leftarrow$
  Delaunay triangulation of vertices of $\Lambda$
  and the scattered points;}
  \FOR{each $q \in V(\mathcal{T}'')$}
  \STATE{Find all $p \in V(\mathcal{T})$ satisfying 
    $q \in (\mathfrak{P}^{\mathcal{T}}_{p})^{-1}(\overline{\Star}(p))$, 
    cf. Definition \ref{def:localProj};}
  \STATE{$q\leftarrow \mathfrak{P}^{\mathcal{T}}_{p}(q)$, 
    where $p$ incurs the shortest projection distance for $q$;}
  \ENDFOR
  \STATE{Apply lines 5--13 of \cref{alg:relocating}
    to $\mathcal{T}''$ with ($\mu,\nu$);}
\IF{$\mathcal{T}''$ satisfies
  the $(r_{\mathrm{tiny}},h_{L},\theta)$-regularity and
  the minimum angle of $\mathcal{T}''$ is greater than that of
  $\mathcal{T}'$ }
\STATE{$\mathcal{T}' \leftarrow \mathcal{T}''$, isRegular $\leftarrow$ True;}
  \ENDIF
  \ENDFOR
  \ENDFOR
  \IF{isRegular $==$ True}
  \STATE{Replace $\mathcal{T}$ in $\mathcal{T}_{\partial \mathcal{M}}$
    by $\mathcal{T}'$;}
  \ENDIF
 \STATE{iter $\leftarrow$ iter+1;}
  \ENDWHILE
\end{algorithmic}
\end{algorithm}




\section{Tests}
\label{sec:tests}
In this section,
 we test the proposed method by several benchmark problems
 and compare its results with those of other IT methods
 in the literature.
 
In most tests, the initial Yin set   
 is a sphere approximated by a triangular mesh
 with uniformly distributed markers. 
Advected by the flow map of a smooth velocity field,
 the triangular mesh undergoes nontrivial deformations
 during the simulation
 and returns to its initial position at the final time.
All flow maps are approximated
 by the classic fourth-order Runge-Kutta method,
 which is chosen solely based on ease of implementation;  
the convergence rates of our 3D linear MARS method
 would be the same
 if another second- or third-order time integrator were employed.

As the grid size is reduced, 
 computing IT errors
 by the symmetric difference in (\ref{equ:totalerror})
 tends to be more and more ill-conditioned.
Instead, in our tests we calculate 
\begin{equation}
  \label{equ:testerrornorms}
  \begin{array}{c}
    E_1(T) := \frac{1}{N(T)}
    \sum_{i=1}^{N(T)} \vert D_i(T)-R\vert, 
  \end{array}
\end{equation}
where $T$ is the final time of simulation,
$R$ the exact radius of the sphere at $T$,
$N(T)$ the number of markers, 
and $D_i(T)$  the distance between the $i$th marker
and the sphere center $C$.
To compare our results with those of other IT methods
 such as VOF or MOF methods,
 we also calculate a geometric error as
\begin{equation}
  \label{eq:geometricIT-error}
  E_{g}(T) := 4\pi R^2 E_{1}(T), 
\end{equation}
where $4\pi R^2$
 is the exact surface area of the sphere at $T$.
Whenever the exact final solution is a sphere,
 (\ref{eq:geometricIT-error}) serves
 as a well-conditioned discrete counterpart of \cref{equ:totalerror}.

\subsection{Vortical shear of a sphere}
\label{subsec:shear}

Referring to \Cref{def:it}, 
 the flow map of this test \cite{LiRu06} is that of the ODE
 $\frac{\dif \mathbf{x}}{\dif t}=\mathbf{u}(\mathbf{x},t)$
 with $\mathbf{u}:=(u_x,u_y,u_z)^{\top}$ 
 given by 
\begin{equation}
  \label{equ:vorticalflow}
  \left\{
    \begin{array}{l}
      u_{x}(x, y, z, t) = 2\sin ^2(\pi x)\sin(2\pi y) 
      \cos\left(\frac{t}{T} \pi\right);\\
      u_{y}(x, y, z, t) = -\sin(2\pi x) \sin ^2(\pi y)
      \cos\left(\frac{t}{T} \pi\right); \\
      u_{z}(x, y, z, t) = (1-2r)^{2}
      \cos\left(\frac{t}{T}\pi\right),
    \end{array}
  \right. 
\end{equation}
 where $r = \sqrt{(x-0.5)^2+(y-0.5)^{2}}$
 is the distance from the vortex axis.
Values of other parameters for this test
 are displayed in \cref{tab:vorticalshearset}.
At time $t=\frac{T}{2}$,
 the velocity field is reversed by the cosinusoidal factor
 $\cos\left(\frac{t}{T}\pi\right)$
 so that the exact final solution at $T = 3$
 coincides with the initial sphere.

 \begin{table}
   \centering
   \caption{Values of parameters for the vortical shear of a sphere.}
   \label{tab:vorticalshearset}
   \begin{tabular}{c|c}
     \hline
     Parameters & Values\\
     \hline
     computational domain & $\Omega = \left[ 0,1 \right]^{3}$\\
     initial/final sphere
                & $C = (0.5, 0.75, 0.25)^{\top}$; $R = 0.15$\\
     velocity field $\mathbf{u}$ 
     & that in (\ref{equ:vorticalflow}) \\
     simulation interval  & $ t\in [ 0, T]$; $T=3$ \\
     Eulerian grid size & $h=\frac{1}{32}, \frac{1}{64}, \frac{1}{128}$\\
     uniform time step size $k$ 
                & $\mathrm{Cr} := \frac{\|\mathbf{u}\|_{\infty}}{h}k = 0.5$
     \\
     $(r_{\mathrm{tiny}},h_{L},\theta)$-regularity
                & $ r_{\mathrm{tiny}}=0.1$; 
                  $h_{L} = 0.5 h, 6h^{\frac{3}{2}}$; 
                  $\theta = \frac{\pi}{10}$
     \\
     VREM & $\mu = 4$; $\nu = 10$; $c= 10^{-4}$; $\rho=0.8$\\
     LTR & $\mu = 4$; $\nu = 10$; $\eta = 3$\\
     \hline
   \end{tabular}
 \end{table}

\begin{figure}
  \label{fig:vorticalshear}
  \centering
  \subfigure[$t=0$]{\includegraphics[height=0.25\textwidth]{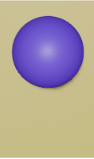}}
  \subfigure[$t=0.6$]{\includegraphics[height=0.25\textwidth]{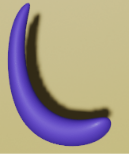}}
  \subfigure[$t=1.5$]{\includegraphics[height=0.25\textwidth]{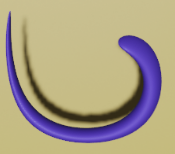}}
  \subfigure[$t=T=3$]{\includegraphics[height=0.25\textwidth]{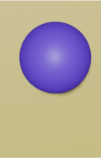}}
  \caption{Results of the linear MARS method
    for solving the vortical shear test
    in \Cref{tab:vorticalshearset} 
    with $h = \frac{1}{64}$ and $h_L=0.5h$.
    The distances between adjacent markers at the initial time $t=0$
    are set to the uniform constant $0.25h$.}
\end{figure}

\begin{table}
  \label{tab:vorticalshearres}
\centering
\caption{IT errors and convergence rates of the 3D linear MARS method
  and those of VOF methods in solving the vortical shear tests
  in \Cref{tab:vorticalshearset}. 
  Results of VOF methods are taken from
  \cite[Tab. 5]{LiRu06} and \cite[Tab. 2]{CoSp21}. 
  In each of the six names,
  the acronyms before and after the slash
  refer to its advection and reconstruction schemes, respectively; 
  CCU, MYC, LSF, EDS, PCFSC and CVTNA stand for
  the cellwise conservative unsplit scheme \cite{CoSp21},
  the mixed Youngs-centered scheme \cite{AuMa07},
  the least-square fit method \cite{ScZa03},
  the Eulerian directional-split scheme \cite{LiRu06},
  the piecewise-constant flux surface calculation scheme
  \cite{LiRu06},
  and the centroid-vertex triangle-normal averaging scheme \cite{LiRu06}, respectively.
}
\begin{tabular}{c|ccccc}
  \hline \hline
  Results based on $E_{g}$ in (\ref{eq:geometricIT-error})
  & $h=\frac{1}{32}$ & rate & $h=\frac{1}{64}$ & rate &  $h=\frac{1}{128}$
  \\
  \hline 
  linear MARS ($h_{L} = 0.5h$) &    1.44e-4 &       2.01 &   3.60e-5  & 2.01 &   8.93e-6 \\
  linear MARS ($h_{L} = 6h^{\frac{3}{2}}$) & 5.68e-4 &       2.93 &   7.43e-5 &       3.02 &   9.16e-6  \\
  \hline
  CCU/Youngs \cite{CoSp21} & 3.76e-3 & 1.72 & 1.14e-3 & 1.73 & 3.45e-4 \\
  CCU/MYC \cite{CoSp21} & 3.71e-3 & 1.81 & 1.06e-3 & 1.99 & 2.66e-4 \\
  CCU/LSF \cite{CoSp21} & 1.68e-3 & 1.88 & 4.57e-4 & 1.37 & 1.77e-4   \\
  EDS/Youngs \cite{LiRu06} & 3.42e-3 &1.58&1.14e-3&1.60&3.76e-4\\
  PCFSC/Youngs \cite{LiRu06} &3.39e-3&1.60&1.12e-3&1.63&3.61e-4\\
  PCFSC/CVTNA \cite{LiRu06} &2.86e-3&2.00&7.14e-4&2.19&1.56e-4\\
  \hline \hline
\end{tabular}


\end{table} 

In \cref{fig:vorticalshear},
 we show results of the 3D linear MARS method at several key moments
 during the evolution of the sphere. 
The initial and final spheres are visually indistinguishable. 
Throughout the simulation,
 the topology of the sphere remains the same
 and the $(r_{\mathrm{tiny}}, h_{L}, \theta)$-regularity
 is always fulfilled.
Specifically, at the end of each time step
 all edge lengths are within the interval
 $ \left[\frac{1}{10} h_L, \,h_L \right]$
 and no interior angle of any triangle is less than $\frac{\pi}{10}$.
These results demonstrate 
 the effectiveness of the proposed mesh adjustment algorithms, 
 i.e., VREM in \Cref{subsec:relocate} and LTR in \Cref{subsec:regenerate}.

IT errors and convergence rates of the 3D linear MARS method
 are displayed in \cref{tab:vorticalshearres}, 
 where the second- and third-order convergence rates
 are clearly demonstrated
 for $h_{L}=O(h)$ and $h_L=O(h^{\frac{3}{2}})$, respectively.
It is also demonstrated in \cref{tab:vorticalshearres}
 that our method is more accurate
 than other state-of-the-art IT methods \cite{CoSp21, LiRu06} 
 by at least one order of magnitude.
In particular, 
 the choice $h_{L} = O(h^{\frac{3}{2}})$
 yields third-order IT accuracy
 and an excellent balance between the main flow solver and the IT method.
As discussed at the end of \Cref{subsec:3DlinearMARS}, 
 the complexity of a 3D main flow solver is at least 
 $O(\frac{1}{kh^3}) = O(\frac{1}{h^{4}})$
 while that of a 3D MARS method is $O(\frac{1}{h^{1+2\alpha}})$. 
Since the operations per control volume
 in advancing the Navier-Stokes equations
 are more complex than those of IT,
 the CPU time consumed by the third-order MARS method
 with $\alpha=\frac{3}{2}$
 is dominated by that of the main flow solver. 

To sum up,
 the proposed MARS method is highly accurate,
 is flexible of choosing $\alpha$ in $h_L=O(h^{\alpha})$, 
 and preserves topological structures
 and geometric features with ease.
These advantages are natural consequences of our principle
 of handling topological and geometric problems
 with tools in topology and geometry. 

\subsection{Deformation of a sphere}

The details on the setup of this test \cite{LeVe96} 
 are summarized in \Cref{tab:deformationTestSetup},
 where the velocity field
 $\mathbf{u}=(u_x,u_y,u_z)^{\top}$ is given by 
\begin{equation}
  \label{eq:deformationVel}
  \left\{
    \begin{array}{l}
      u_{x}(x, y, z, t) =\hspace{0.8mm} 2\sin ^2(\pi x)\sin(2\pi y) \sin(2\pi z)
      \cos\left(\frac{t}{T}\pi\right);\\
      u_{y}(x, y, z, t) = -\sin ^2(\pi y)\sin(2\pi z) \sin(2\pi x)
      \cos\left(\frac{t}{T}\pi\right);\\
      u_{z}(x, y, z, t) = -\sin ^2(\pi z)\sin(2\pi x) \sin(2\pi y)
      \cos\left(\frac{t}{T}\pi\right)
    \end{array}
  \right.  
\end{equation}
and the cosinusoidal factor
reverses the direction of the velocity field at $t=\frac{T}{2}$
so that the Yin set at $t=T$ is exactly the same as
the initial ball at $t=0$.

\begin{table}
  \centering
  \caption{Values of parameters for the deformation of a sphere.}
  \label{tab:deformationTestSetup}
  \begin{tabular}{c|c}
    \hline
    Parameters& Values\\
    \hline
    computational domain & $\Omega = \left[ 0,1 \right]^{3}$\\
    initial/final sphere
              & $C = (0.35, 0.35, 0.35)^{\top}$; $R = 0.15$\\
    velocity field $\mathbf{u}$ 
              & that in (\ref{eq:deformationVel}) \\
    simulation interval  & $ t\in [ 0, T]$; $T=3$ \\
    Eulerian grid size & $h=\frac{1}{32}, \frac{1}{64},
                         \frac{1}{128}, \frac{1}{256}$\\
    uniform time step size $k$ 
              & $\mathrm{Cr} := \frac{\|\mathbf{u}\|_{\infty}}{h}k = 0.5$
    \\
    $(r_{\mathrm{tiny}},h_{L},\theta)$-regularity
              & $ r_{\mathrm{tiny}}=0.1$; 
                $h_{L} = 0.5 h, 6h^{\frac{3}{2}}$;
                $\theta = \frac{\pi}{10}$
    \\
    VREM & $\mu = 4$; $\nu = 10$; $c= 10^{-4}$; $\rho=0.8$\\
     LTR & $\mu = 4$; $\nu = 10$; $\eta = 3$\\
    \hline
  \end{tabular}
\end{table}
 
\begin{table}
  \label{tab:vortexres}
  \centering    
  \caption{IT errors and convergence rates of the 3D linear MARS method
    and VOF/MOF methods in solving the deformation tests
    in \Cref{tab:deformationTestSetup}.
    Results of the six geometric VOF methods are taken from
    \cite[Tab. 7]{LiRu06} and \cite[Tab. 3]{CoSp21}
    while those in the last three lines from \cite[Tab. 3]{AsMe20} and
    \cite[Tab. 6]{MuMe22},
    where the acronyms MOF, CLSVOF, HyMOFLS
    stand for the moment-of-fluid method \cite{AhSh09},
    the coupled level set and volume-of-fluid method \cite{SuPu00},
    and the hybrid moment of fluid--level set method \cite{MuMe22}.}
  \small
   \begin{tabular}{c|ccccccc}
     \hline \hline
   Results based on $E_{g}$ 
     & $h=\frac{1}{32}$ & rate & $h=\frac{1}{64}$ & rate &  $h=\frac{1}{128}$
     & rate &  $h=\frac{1}{256}$\\
     \hline 
   linear MARS ($h_{L} = 0.5h$)
     &   1.48e-4 &     2.00& 3.73e-5&     1.99 & 9.42e-6& 2.00&
  2.35e-6\\
   linear MARS ($h_{L} = 6h^{\frac{3}{2}}$)
     & 6.33e-4 & 2.99 & 7.97e-5 &  2.97 & 1.02e-5 &3.01&
  1.27e-6\\
      \hline
   CCU/Youngs \cite{CoSp21}
  & 7.21e-3 & 1.46 & 2.63e-3 & 1.97 & 6.73e-4 & 1.65 & 2.14e-4 \\
  CCU/MYC \cite{CoSp21}
  & 6.94e-3 & 1.50 & 2.46e-3 & 2.26 & 5.13e-4& 2.58 & 8.59e-5  \\
  CCU/LSF \cite{CoSp21}
  & 6.11e-3 & 1.86 & 1.68e-3 & 2.22 & 3.59e-4 & 2.25 & 7.54e-5  \\
  EDS/Youngs \cite{LiRu06}&7.71e-3&1.47&2.78e-3&1.87&7.58e-4&1.68&2.37e-4\\
  PCFSC/Youngs \cite{LiRu06} &7.86e-3&1.43&2.91e-3&1.98&7.36e-4&1.70&2.26e-4\\
   PCFSC/CVTNA \cite{LiRu06} &7.41e-3&1.90&1.99e-3&2.69&3.09e-4&2.14&7.03e-5\\
   \hline
   MOF \cite{AsMe20} & 5.62e-3 &1.26&2.35e-3&2.13&5.38e-4&1.58&1.80e-4\\
   \hline
  CLSVOF \cite{AsMe20}
     & 7.77e-3 & 1.15 & 3.50e-3 & 2.18 & 7.72e-4& 1.72 & 2.35e-4  \\
   \hline
   HyMOFLS \cite{MuMe22}
   &5.80e-3&1.57&2.30e-3&2.14&5.23e-4&1.45&1.92e-4\\
     \hline \hline
   \end{tabular}


\end{table}

\begin{figure}
  \label{fig:vortex}
  \centering
  \subfigure[$t=0$]{\includegraphics[height=0.301\textwidth]{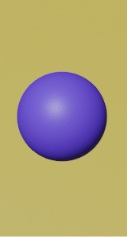}}
  \subfigure[$t=0.375$]{\includegraphics[height=0.301\textwidth]{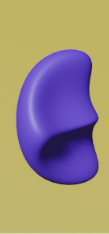}}
  \subfigure[$t=0.75$]{\includegraphics[height=0.301\textwidth]{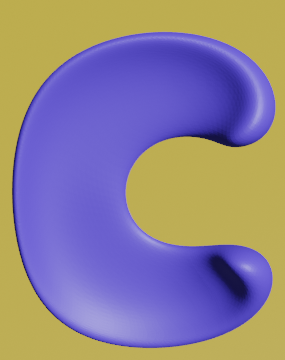}}
  \subfigure[$t=1.5$]{\includegraphics[height=0.301\textwidth]{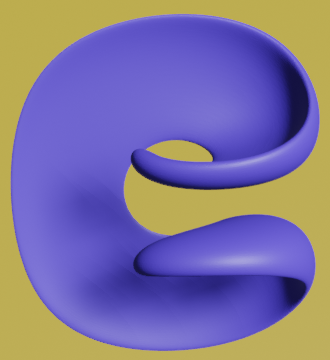}}
  \subfigure[$t=T=3$]{\includegraphics[height=0.301\textwidth]{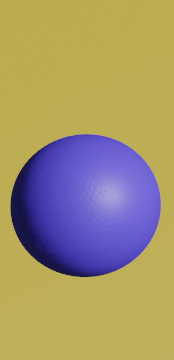}}
  \caption{Results of the 3D linear MARS method
    in solving the deformation test in \Cref{tab:deformationTestSetup}
    with $h = \frac{1}{64}$ and $h_L=0.5h$.
    The initial distances between adjacent markers
    are set to the uniform $0.25h$.}
\end{figure}

As shown in \cref{fig:vortex}(d),
 the initial sphere deforms into a very thin film at $t=\frac{T}{2}$, 
 posing a challenge to VOF and MOF methods
 in preserving topological structures and geometric features.
As demonstrated in \cite[Fig. 8]{BoSc22}, 
 VOF and MOF methods tend to develop holes
 on the thin film,
 altering the topology;  
 the final results of these methods also have
 geometric features that are prominently different
 from those of the initial sphere.
In comparison,
 at each time step of the 3D linear MARS method,
 the evolving Yin set remains homeomorphic to the initial sphere,
 preserving the topological invariant; 
 also, the initial and final spheres in \cref{fig:vortex}(a,e)
 are visually indistinguishable.
Furthermore,
 the triangular mesh of the 3D linear MARS method
 at the end of each time step 
 always satisfies 
 the $(r_{\mathrm{tiny}},h_{L},\theta)$-regularity
 specified in \Cref{tab:deformationTestSetup}.

Errors and convergence rates of our 3D MARS method
 are displayed in \cref{tab:vortexres},
 where the second- and third-order convergence rates
 of our method are clearly demonstrated
 for the choices of $h_{L}=O(h)$ and $h_L=O(h^{\frac{3}{2}})$,
 respectively.
We also collect in \cref{tab:vortexres}
 errors and convergence rates
 of other state-of-the-art IT methods
 available in the literature \cite{CoSp21, AsMe20, LiRu06}. 
On the finest grid,
 the proposed second-order linear MARS method
 is more accurate than other IT methods
 by a factor of at least 30.

\subsection{Deformation of an armadillo
   with complex geometry}
 \label{sec:deform-an-armad}

This test is the same as that in \Cref{tab:deformationTestSetup}
 except that the sphere at the initial time $t=0$
 has been replaced by an armadillo with complex geometry
 in \Cref{fig:armadillo}(a),
 of which the triangular mesh is obtained
 from the Stanford 3D
 scanning repository \cite{StAr}.  
 
 \begin{figure}
  \label{fig:armadillo}
  \centering
  \subfigure[$t=0$]{\includegraphics[width=0.312\textwidth]{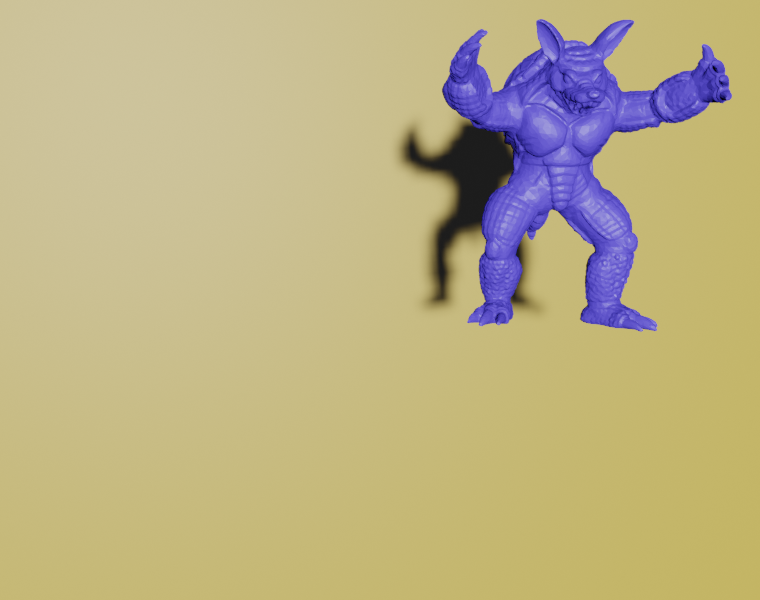}}
  \subfigure[$t=0.3$]{\includegraphics[width=0.312\textwidth]{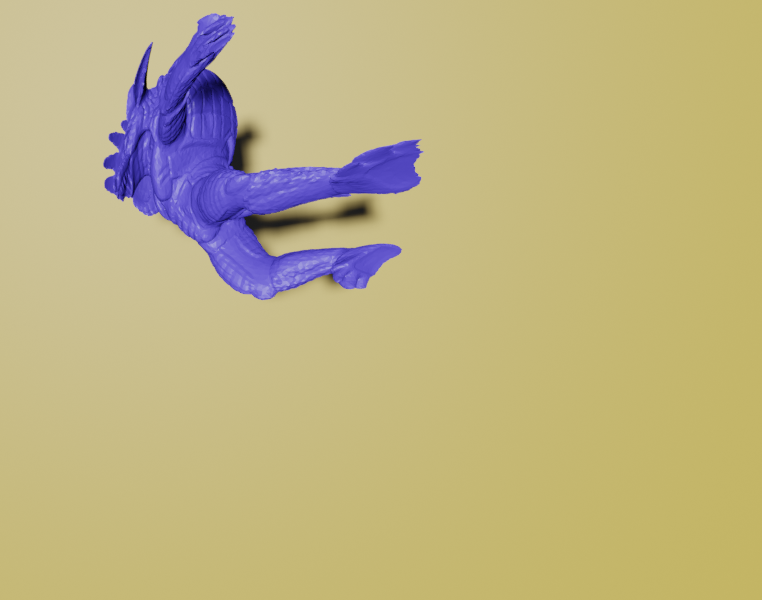}}
  \subfigure[$t=0.6$]{\includegraphics[width=0.312\textwidth]{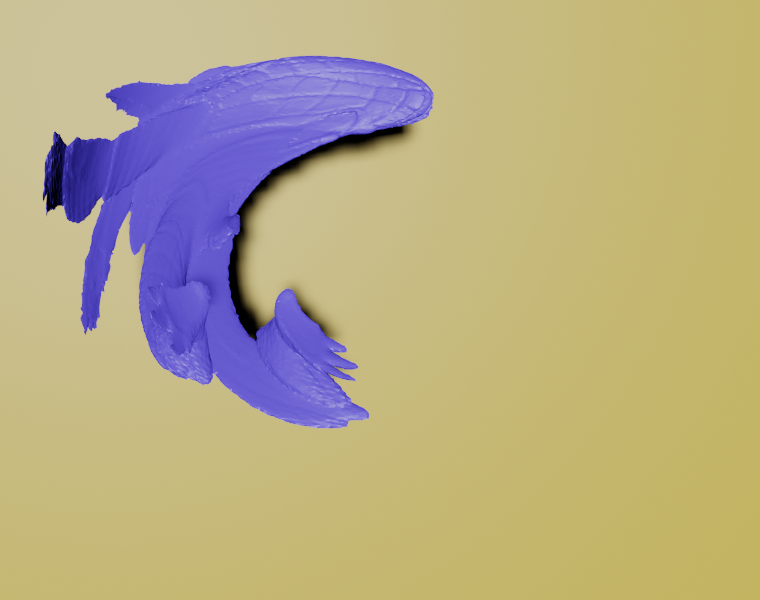}}
  
 \subfigure[$t=0.9$]{\includegraphics[width=0.312\textwidth]{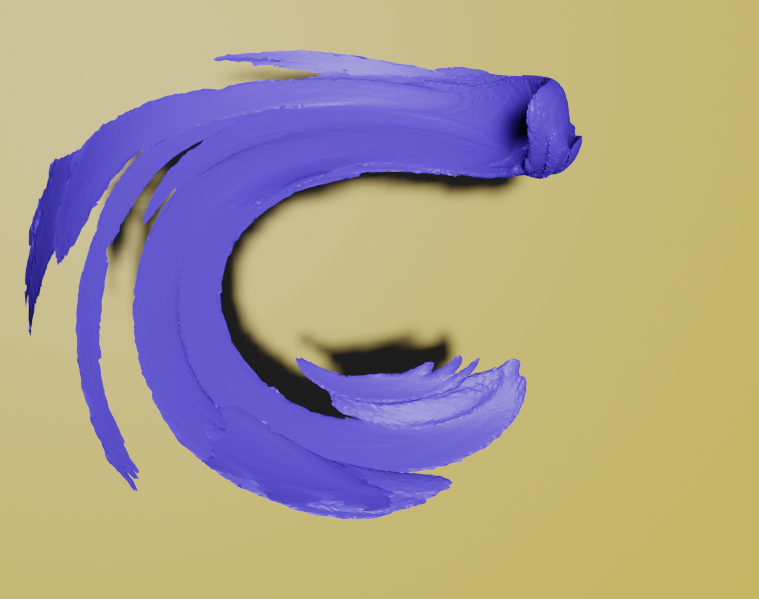}}
  \subfigure[$t=\frac{T}{2}=1.5$]{\includegraphics[width=0.312\textwidth]{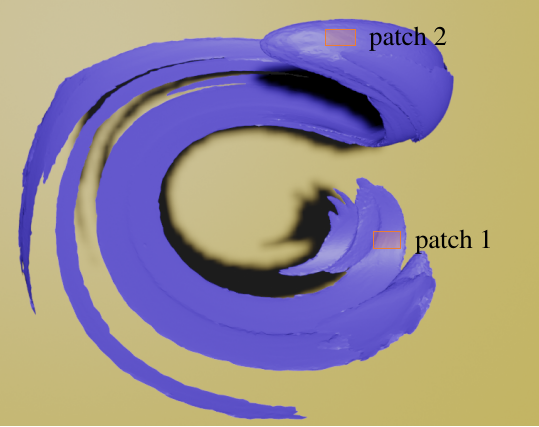}}
  \subfigure[$t=T=3$]{\includegraphics[width=0.312\textwidth]{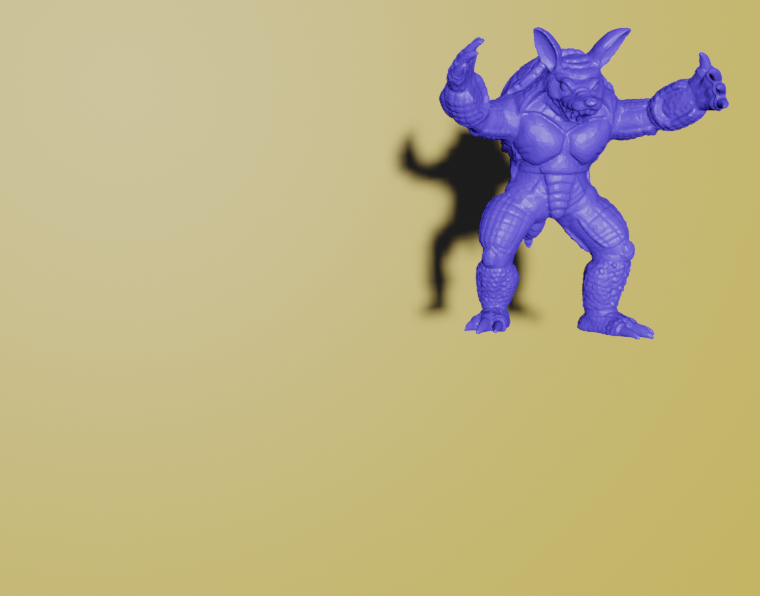}}
  \caption{Results of the 3D linear MARS method
    in solving the deformation test in \Cref{tab:deformationTestSetup}
    with $h = \frac{1}{128}$ and $h_L=0.5h$.
    See \cref{fig:armadillocomparison}(c,d)
    for zoom-in of patch 1 and patch 2 in subplot (e).
  }
\end{figure} 

\begin{figure}
  \label{fig:armadillocomparison}
  \centering
  \subfigure[patch 1: without VREM and LTR]
  {\includegraphics[width=0.45\textwidth]{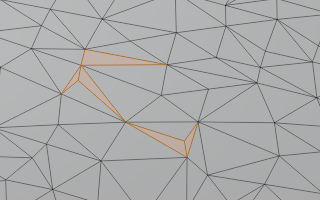}}
  \hfill
  \subfigure[patch 2: without VREM and LTR]
  {\includegraphics[width=0.45\textwidth]{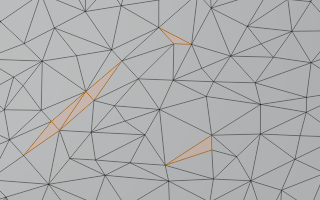}}

  \subfigure[patch 1: with VREM and LTR]
  {\includegraphics[width=0.45\textwidth]{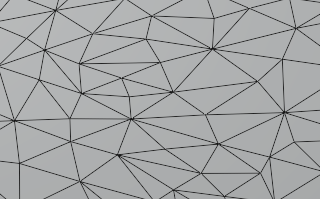}}
  \hfill
  \subfigure[patch 2: with VREM and LTR]
  {\includegraphics[width=0.45\textwidth]{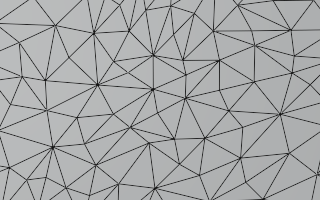}}
  \caption{Comparing mesh quality of the two surface patches
    marked in \cref{fig:armadillo}(e)
    for the armadillo at $t=\frac{T}{2}$. 
    Results in (c,d) are obtained via the linear MARS method
    by using all three mesh adjustments
    in Subsections \ref{subsec:elementary},
    \ref{subsec:relocate},
    and \ref{subsec:regenerate}, 
    while those in (a,b) by turning off VREM and LTR. 
    The minimum angles in (a) and (b) are respectively 10.3° and 14.1°
    while those in (c) and (d) are respectively 23.6° and 29.8°.
    The highlighted triangles in (a,b) 
    violate the $\theta$-regularity.
  } 
\end{figure}

The simulation results of our 3D MARS method
 are shown in \cref{fig:armadillo}, 
 where subplots (c--e) demonstrate
 significant distortions of the armadillo.
Nonetheless,
 the topological invariants of the armadillo
 and the $(r_{\mathrm{tiny}},h_{L},\theta)$-regularity
 in \Cref{tab:deformationTestSetup}
 are well preserved at each time step. 
In addition, the geometry of the final armadillo
 is visually indistinguishable from that of the initial condition, 
 cf. \cref{fig:armadillo}(a,f).

With both VREM and LTR turned off,
 we perform the armadillo test once again, 
 locate at $t=\frac{T}{2}$
 some ill-shaped triangles that violate the $\theta$-regularity, 
 and plot in \cref{fig:armadillocomparison} 
 two surface patches that contain some of these ill-shaped triangles.
Also displayed in \cref{fig:armadillocomparison}  
 are the corresponding results of our MARS method
 with both VREM and LTR turned on;
 they are subsets of the triangular mesh in \Cref{fig:armadillo}(e).
A comparison of the subplots in \cref{fig:armadillocomparison} 
 clearly demonstrates the efficacy of VREM and LTR: 
 without them the minimum angle is $10.3^{\circ}$
 while turning them on increases the minimum angle
 to $23.6^{\circ}$.

 \begin{table} 
   \centering
   \caption{Cost balance results of EMA, VREM, and LTR
     with $h_{L}=0.5h$
     for the vortical shear tests
     in \cref{tab:vorticalshearset}
     and the deformation tests
     in \cref{tab:deformationTestSetup}. 
   }
   \label{tab:prop}
   \small  
   \begin{tabular}{l|ccc}
  \hline\hline
  \multicolumn{4}{l}{(a): percentages 
  in terms of the number of adjusted triangles}\\
  \hline
   Test case & EMA & VREM & LTR\\
    \hline
  vortical shear of the sphere with $h=\frac{1}{32}$
  & 90.3\% &8.9\% &0.8\%  \\
    \hline
  vortical shear of the sphere with $h=\frac{1}{64}$
  & 90.8\% & 8.4\% &0.8\%\\
    \hline
  vortical shear of the sphere with $h=\frac{1}{128}$
  & 90.5\% & 8.5\% & 1.0\%\\
  \hline
  deformation of the sphere with $h=\frac{1}{64}$
  & 93.4\% & 6.0\% &0.6\%\\
    \hline
  deformation of the sphere with $h=\frac{1}{128}$
  & 93.4\% & 5.9\% & 0.7\%\\
  \hline
   deformation of the sphere with $h=\frac{1}{256}$
  & 93.3\% & 6.0\% & 0.7\%\\
  \hline
  deformation of the armadillo with $h=\frac{1}{128}$
   & 89.8\% & 7.9\% & 2.3\%\\
  \hline
  \hline
  \multicolumn{4}{l}{(b): percentages 
  in terms of consumed CPU time}\\
  \hline
  Test case & EMA & VREM & LTR\\
    \hline
  vortical shear of the sphere with $h=\frac{1}{32}$
  & 31.6\% &32.9\% &35.5\%  \\
    \hline
  vortical shear of the sphere with $h=\frac{1}{64}$
  & 32.3\% & 32.5\% &35.2\%\\
    \hline
  vortical shear of the sphere with $h=\frac{1}{128}$
  & 31.3\% & 33.0\% & 35.7\%\\
  \hline
  deformation of the sphere with $h=\frac{1}{64}$
  & 36.5\% & 30.8\% &32.7\%\\
    \hline
  deformation of the sphere with $h=\frac{1}{128}$
  & 36.6\% & 31.3\% & 32.1\%\\
  \hline
   deformation of the sphere with $h=\frac{1}{256}$
  & 35.5\% & 31.7\% & 32.8\%\\
  \hline
  deformation of the armadillo with $h=\frac{1}{128}$
   & 26.0\% & 27.4\% & 46.6\%\\
  \hline
  \hline 
\end{tabular}

 \end{table}

\subsection{Balancing costs
  of the three mesh adjustment algorithms}
\label{subsec:proportion}

A question that follows naturally from \cref{fig:armadillocomparison}
 is: how expensive are VREM and LTR?
To answer this question,
 we list in \cref{tab:prop} percentages of EMA, VREM, and LTR
 in terms of both the number of adjusted triangles
 and the consumed CPU time.

As shown in part (a) of \cref{tab:prop}, 
 about 90\% of ill-shaped triangles
 are handled by the three fast and simple operations in EMA
 while less than 2.5\% of them by 
 the most expensive and complex mesh adjustments in LTR. 
Although their combined percentage 
 is less than one ninth of the percentage of EMA,
 VREM and LTR are essential in ensuring 
 the $(r_{\mathrm{tiny}},h_{L},\theta)$-regularity
 of the evolving triangular mesh. 

In part (b) of \cref{tab:prop},
 the percentages of EMA, VREM, and LTR
 in terms of consumed CPU time
 are roughly the same for all sphere tests,
 indicating
 that the cost ratio of VREM--LTR
 over EMA is bounded by 2.2. 
For the armadillo test, 
 the ratio is about 3.
Therefore,
 the total cost of the sequence
 EMA--VREM--LTR is satisfactory to a high degree,  
 because the $(r_{\mathrm{tiny}},h_{L},\theta)$-regularity
 is fulfilled by a cost
 only several times greater than that of EMA.
As the key of this success,
 the first-applied EMA is the simplest/fastest
 and handles most ill-shaped triangles
 while the last-applied LTR, despite its complexity,
 is invoked on the smallest fraction of ill-shaped triangles.



\section{Conclusions}
\label{sec:conclusions}

For 3D explicit IT,
 we have developed a linear MARS method, 
 which is fundamentally different from all current IT methods
 in that topological and geometric problems
 are tackled by tools in topology and geometry. 
However complex it may be,
 any 3D continuum is uniquely represented by a poset of glued surfaces
 and each glued surface can be well approximated by a triangular mesh.
By designing two novel local mesh adjustment algorithms, 
 we maintain, at each time step, 
 the $(r_{\mathrm{tiny}},h_{L},\theta)$-regularity
 on the triangular meshes. 
Numerical results of benchmark tests demonstrate 
 the second- or third-order convergence rates of our method
 and its advantages
 in preserving topological structures and geometric features.

Future research prospects follow.
First, we plan to analyze the proposed method
 under the MARS framework.
Second, we will develop fourth- and higher-order 3D IT methods
 by representing the interface {topology} 
 with a triangular mesh
 and by approximating the interface {geometry}
 with spline surfaces.
Third, we will couple MARS methods 
 with high-order flow solvers
 to simulate flows with moving boundaries. 




\section*{Acknowledgments}
We acknowledge helpful comments
 from Junxiang Pan, Zhiqi Li, Chenhao Ye, Lei Pang, and Shuang Hu, 
 graduate students at the school of mathematical sciences
 in Zhejiang University.

\bibliographystyle{siamplain}
\bibliography{bib/MARS-3D-linear}
\end{document}